\documentclass[10pt,reqno,oneside]{amsart}
\usepackage{amsmath}
\usepackage{amsthm}
\usepackage{newtxtext,newtxmath}
\usepackage{ulem}
\usepackage{eqparbox}
\usepackage{mathtools}
\usepackage{array}
\usepackage[bb=ams,bbscaled=.97,cal=cm,calscaled=.96,frak=euler]{mathalfa}
\usepackage[utf8]{inputenc}
\usepackage{graphicx}
\usepackage{stmaryrd}

\usepackage[textwidth=125mm,textheight=195mm,paperwidth=210mm,paperheight=297mm]{geometry}

\newcommand{\hair}{\ifmmode\mskip1.5mu\else\kern0.05em\fi}

\usepackage{enumitem}
\setenumerate[0]{label=(\textit{\roman*}\hair),font=\rm}

\usepackage[colorlinks]{hyperref}
\hypersetup{
 pdftitle = {},
 pdfauthor = {},
}
\usepackage{tikz}

\DeclareMathOperator{\rank}{rank}

\usepackage{tikz}
\usetikzlibrary{matrix,arrows,decorations.markings,calc,intersections,shapes,backgrounds}

\newcommand{\llrr}[1]{%
\llbracket #1 \rrbracket}
\newcommand{\blank}{\mkern3mu\cdot\mkern3.7mu}
\newcommand{\hatotimes}{\mathbin{\widehat\otimes}}

\numberwithin{equation}{section}

\newtheorem{theorem}[equation]{Theorem}
\newtheorem*{theorem*}{Theorem}
\newtheorem{proposition}[equation]{Proposition}
\newtheorem{lemma}[equation]{Lemma}
\newtheorem*{lemma*}{Lemma}
\newtheorem{corollary}[equation]{Corollary}
\newtheorem*{corollary*}{Corollary}
\theoremstyle{definition}
\newtheorem{definition}[equation]{Definition}
\newtheorem{example}[equation]{Example}
\newtheorem{notation}[equation]{Notation}
\theoremstyle{remark}
\newtheorem{remark}[equation]{Remark}

\newcommand{\Wedge}{\mathrm{\Lambda}\mkern-2mu}
\DeclareMathOperator{\Dgn}{dgn}
\DeclareMathOperator{\Ext}{Ext}

\DeclareMathOperator{\Pic}{Pic}

\DeclareMathOperator{\GL}{GL}
\DeclareMathOperator{\SL}{SL}
\DeclareMathOperator{\SU}{SU}

\DeclareMathOperator{\Tot}{Tot}

\DeclareMathOperator{\Sym}{S}

\DeclareMathOperator{\Spec}{Spec}

\DeclareMathOperator{\tr}{tr}

\DeclareMathOperator{\id}{id}
\newcommand{\HH}{\operatorname{HH}}

\renewcommand{\hbar}{\text{\raisebox{.5ex}{\rotatebox{15}{--}}\hspace{-0.6em}$h$}}

\renewcommand{\H}{\operatorname{H}}
\usepackage{xfrac}

\begin{document}
\title[Quantizations and noncommutative instantons]{Quantizations of local surfaces and rebel instantons}

\author{Severin Barmeier}
\author{Elizabeth Gasparim}

\begin{abstract}
We construct explicit deformation quantizations of the noncompact complex surfaces $Z_k := \Tot (\mathcal O_{\mathbb P^1} (-k))$ and describe their effect on moduli spaces of vector bundles and instanton moduli spaces. We introduce the concept of rebel instantons, as being those which react badly to some quantizations, misbehaving by shooting off extra families of noncommutative instantons. We then show that the quantum instanton moduli space can be viewed as the étale space of a constructible sheaf over the classical instanton moduli space with support on rebel instantons.
\end{abstract}

\maketitle

\setcounter{tocdepth}{1}
\tableofcontents


\section{Motivation and results}

In this work we clarify a specific aspect of the quantization of $\SU(2)$ instantons, by describing explicitly the effect that deformation quantization has on moduli spaces of instantons over noncompact complex surfaces. Our main contribution is to identify those classical instantons which react badly to certain choices of quantization, by shooting off extra families of noncommutative instantons. We call them rebel instantons. Thus, rebel instantons cause the quantum moduli spaces to become larger than the corresponding moduli on the classical limit.

Our complex surfaces of choice are total spaces of negative line bundles on the complex projective line $\mathbb P^1$, namely the surfaces $Z_k := \Tot(\mathcal O_{\mathbb P^1} (-k))$ for $k \geq 1$.  Our interest in them originates in the wish to understand how the mathematical physics on them reacts to contraction of the complex curve $\mathbb P^1$ to a point. This can only be done in the case of negative normal bundle, hence the choice of $-k$ neighbourhoods. Contracting this curve produces a singular surface whenever $k>1$. In fact, the case $k=1$ --- where the contraction produces the smooth surface $\mathbb C^2$ --- allows for quantization without any rebel instantons, in contrast to what happens when $k > 1$.

For each surface $Z_k$ we compute all possible holomorphic Poisson structures, 
and then calculate explicitly corresponding deformation quantizations, writing out star products. We study instantons and their moduli via the Kobayashi--Hitchin 
correspondence, that is, by constructing holomorphic vector bundles over these surfaces. We describe vector bundles concretely using matrices, in the spirit of the ADHM construction, but in a somewhat further simplified manner which is made possible by the filtrability of bundles on $Z_k$. Accordingly, an instanton or vector bundle on $Z_k$ can be described by a single upper-triangular matrix with polynomial entries. A similar use of filtrability allows us to describe vector bundles over  noncommutative deformations and to describe how their moduli spaces change after deforming the surface. 

Our main result is the following:

\begin{theorem*}[Thm.~\ref{qi}]
The quantum instanton moduli space $\mathbb{QI}_j^{(1)} (\mathcal Z_k (\sigma))$ can be viewed as the étale space of a constructible sheaf over the classical instanton moduli space $\mathbb {MI}_j( Z_k)$, which is supported on a closed subvariety, being trivial over 
$$
S_0:= \{P \in \mathbb{MI}_j (Z_k) \mid p_{1,k-j+1} \neq 0 \}
$$
and having stalk of dimension $i$ over 
$$
S_i := \{P \in \mathbb {MI}_j (Z_k) \mid p_{1,k-j+1} = \cdots = p_{1,k-j+1+i} = 0, \; p_{1,k-j+2+i} \neq 0 \}.
$$
\end{theorem*}

Since we are dealing with moduli space of bundles on non-compact varieties, the topology of moduli spaces of bundles is rather subtle, already for the classical moduli space \cite{BGK2}. A detailed analysis of the topology of the quantum moduli space would certainly be important and might provide insight into further interesting phenomena, such as studying the effect of evaluating the formal deformation parameter $\hbar$ to a non-zero constant where possible. However, these considerations would deserve a separate treatment and we thus choose to sideline these issues for the present article by simply viewing the quantum moduli space as the étale space (or sheaf space) of a constructible sheaf, constant over each stratum, making the map to the classical moduli space continuous.

Nonetheless Thm.~\ref{qi} already shows that the effect of noncommutative deformations on instantons is radically different from the effect of commutative ones. In fact, \cite[Thm.\ 7.3]{BrG} showed that (nontrivial) classical deformations of $Z_k$ admit {\it no} instantons, although the surfaces $Z_k$ have rich instanton moduli spaces. (See Not.~\ref{insrep} for the representation of instantons in canonical coordinates.)

We now describe some of the literature on the subject and next the structure of this paper. Instantons on noncommutative spaces have been considered from various points of view, starting with instantons on noncommutative $\mathbb R^4$ and noncommutative tori \cite{NS} and their relations with string theory \cite{SW}. Nekrasov--Schwarz \cite{NS} proposed a modification of the ADHM equations, which describe instantons on a noncommutative $\mathbb R^4$. Kapustin--Kuznetsov--Orlov \cite{KKO} showed that the complex point of view can also be generalized to the noncommutative setting, identifying these noncommutative instantons with algebraic vector bundles on a noncommutative projective plane $\mathbb P^2_\hbar$ framed at a line at infinity. The space of solutions to these modified ADHM equations turns out to yield a smooth compactification of the moduli space of instantons on the (commutative) $\mathbb R^4$, so that the noncommutative viewpoint also sheds light onto classical instantons. (Other interesting features and generalizations are as follows. This compactification can be viewed as the moduli space of torsion-free sheaves on $\mathbb P^2$ framed at a fixed line at infinity, see Nakajima \cite{nakajima}. Furthermore,
from the point of view of instanton counting, moduli of torsion-free sheaves were extensively used as partial compactifications of moduli of instantons, in the very successful instanton partition function defined by Nekrasov, and explored in \cite{NO, NY, GL}. There are also approaches to noncommutative instantons from the point of view of noncommutative geometry as for example in \cite{BvS,CLS}. We shall not pursue these aspects here.)

Classical $\SU (2)$ instantons on $Z_k$ can be identified with holomorphic rank $2$ bundles on $Z_k$ via the Kobayashi--Hitchin correspondence, which for $k = 1$ was proven by King \cite{Ki} and for $k \geq 2$ in \cite{GKM}. Considering deformation quantizations $\mathcal A$ of the sheaf $\mathcal O_{Z_k}$ we obtain instantons on noncommutative deformations of $Z_k$ as locally free sheaves of $\mathcal A$-modules, generalizing the concept of holomorphic vector bundle to these noncommutative spaces. In general the surfaces $Z_k$ do not admit locally constant holomorphic Poisson structures, and instead of the Moyal product we shall thus consider the Kontsevich star product extended to these surfaces. 

The paper is organized as follows. In \S\S \ref{poissongeometry}--\ref{quantizableimmersions} we describe the noncommutative deformation theory of $Z_k$. In \S \ref{commutative} we review the theory of vector bundles on $Z_k$ and their commutative deformations and discuss vector bundles on noncommutative deformations of $Z_k$ in \S \ref{vectorbundles}. We use the explicit expressions of star products obtained in \S \ref{quantizableimmersions} to show in \S\ref{ncmoduli} that purely noncommutative deformations have nontrivial moduli of vector bundles. Applications to (noncommutative) instantons are described in \S\S\ref{classicalinstantons}--\ref{instantons}.

\section{Poisson geometry}
\label{poissongeometry}


In this section we study the holomorphic Poisson geometry of the surfaces $Z_k = \Tot \mathcal O_{\mathbb P^1} (-k)$ for $k \geq 1$, which will be the starting point for quantizations in subsequent sections.

\begin{definition}\label{poissonX} \cite{LGPV}
A {\it holomorphic Poisson bracket} on a complex manifold or smooth complex algebraic variety $X$ is a $\mathbb C$-bilinear map
\[
\{ \blank {,} \blank \} \colon \mathcal O_X \times \mathcal O_X \to \mathcal O_X
\]
satisfying
\begin{flalign*}
&& \{ f, g \} &= - \{ g, f \} && \llap{\it (skew-symmetry)} \\
&& \{ f g, h \} &= f \{ g, h \} + \{ f, h \} g && \llap{\it (Leibniz rule)} \\
&& \{ \{ f, g \}, h \} &= \{ f, \{ g, h \} \} + \{ \{ f, h \}, g \} && \llap{\it (Jacobi identity)}
\end{flalign*}
for all $f, g, h \in \mathcal O_X$.\footnote{Here we write $f, g, h \in \mathcal O_X$ for sections of $\mathcal O_X$ over some open set. More precisely, a holomorphic Poisson bracket on $\mathcal O_X$ is a family of Poisson brackets indexed by the open sets of $X$ and compatible with the restriction morphisms of the sheaf $\mathcal O_X$.}
\end{definition}

Equivalently, a holomorphic Poisson structure may be described by a holomorphic bivector field
$\sigma \in \H^0 (X, \Wedge^2 \mathcal T_X)$ whose Schouten--Nijenhuis bracket $[\sigma, \sigma] \in \H^0 (X, \Wedge^3 \mathcal T_X)$ is zero. The associated Poisson bracket is then given by the pairing $\langle \blank {,} \blank \rangle$ between vector fields and forms
\[
\{ f, g \} = \langle \sigma, \mathrm d f \wedge \mathrm d g \rangle.
\]

\begin{remark}
On a smooth {\it surface} $X$ the condition $[\sigma, \sigma] = 0$ is satisfied for any bivector field $\sigma \in \H^0 (X, \Wedge^2 \mathcal T_X)$ since $\Wedge^3 \mathcal T_X = 0$, and thus $\H^0 (X, \Wedge^2 \mathcal T_X)$ may be identified with the space of holomorphic Poisson structures.
\end{remark}

We now focus on the surfaces $Z_k$ and their Poisson structures, which we shall describe explicitly in canonical coordinates.

\begin{notation}\label{coord}
We fix coordinate charts $U, V$ on $Z_k$, which we will
refer to as {\it canonical coordinates}, where
\begin{equation}
\label{charts}
 U = \mathbb C^2_{z,u}   = \bigl\{ (  z, u) \in \mathbb C^2 \bigr\} \qquad\text{and}\qquad
 V = \mathbb C^2_{\xi,v} = \bigl\{ (\xi, v) \in \mathbb C^2 \bigr\}
\end{equation}
such that on $U \cap V = \mathbb C^* \times \mathbb C$ we identify
\begin{align}
\label{identification}
(\xi, v) = (z^{-1}, z^k u)\text{.}
\end{align}

We denote by $\ell$ the ideal of the zero section of $Z_k$ regarded as a divisor. Hence $\ell$ is generated by $u$ on the $U$-chart, and by $v$ on the $V$-chart.
\end{notation}

In canonical coordinates, a holomorphic bivector field $\sigma \in \H^0 (Z_k, \Wedge^2 \mathcal T_{Z_k})$ is of the form $\sigma_U \, \tfrac{\partial}{\partial z} {\wedge} \tfrac{\partial}{\partial u}$ on $U$ and $\sigma_V \, \tfrac{\partial}{\partial \xi} {\wedge} \tfrac{\partial}{\partial v}$ on $V$, where $\sigma_U$ (resp.\ $\sigma_V$) is a holomorphic function in $z$ and $u$ (resp.\ $\xi$ and $v$) whose precise form will be given in Lem.~\ref{poissonstructures}.

Given a bivector field $\sigma$ on $Z_k$ the corresponding {\it Poisson bracket} 
$\{ f, g \}_\sigma$ of two global functions $f, g \in \H^0 (Z_k, \mathcal O)$ may be written in canonical coordinates as
\begin{alignat*}{7}
\{ f, g \}_\sigma \vert_U &= \big\langle &\sigma_U \, \tfrac{\partial}{\partial z} {\wedge} \tfrac{\partial}{\partial u},{}& \mathrm d f_U \wedge \mathrm d g_U \big\rangle
= \sigma_U &&\Bigg( \frac{\partial f_U}{\partial z} \frac{\partial g_U}{\partial u} - &\frac{\partial f_U}{\partial u} \frac{\partial g_U}{\partial z} \Bigg)& \\
\{ f, g \}_\sigma \vert_V &= \big\langle &\sigma_V \, \tfrac{\partial}{\partial \xi} {\wedge} \tfrac{\partial}{\partial v},{}& \mathrm d f_V \wedge \mathrm d g_V \big\rangle
= \sigma_V &&\Bigg( \frac{\partial f_V}{\partial \xi} \frac{\partial g_V}{\partial v} - &\frac{\partial f_V}{\partial v} \frac{\partial g_V}{\partial \xi} \Bigg)&,
\end{alignat*}
where $\mathrm d$ is the exterior derivative and $f_U, g_U$ denote the restrictions of $f,g$ to $U$ written in $(z, u)$-coordinates, and similarly for $f_V, g_V$.

\begin{notation}
When referring to a bivector field $\sigma \in \H^0 (Z_k, \Wedge^2 \mathcal T_{Z_k})$ or its corresponding Poisson structure, we will work in canonical coordinates in the basis $\tfrac{\partial}{\partial z} {\wedge} \tfrac{\partial}{\partial u}$ and $\tfrac{\partial}{\partial \xi} {\wedge} \tfrac{\partial}{\partial v}$ and only write its coefficient functions as a pair $(\sigma_U, \sigma_V)$ --- in fact, we often just write $\sigma_U$ in $(z, u)$-coordinates, as $\sigma_V$ can be recovered by writing $-z^{k-2} \sigma_U$ as a function of $\xi$ and $v$ via the change of variables (\ref{identification}).
\end{notation}

An application of the exponential sheaf sequence
\[
0 \longrightarrow \mathbb Z \longrightarrow \mathcal O \stackrel{\exp}\longrightarrow \mathcal O^* \longrightarrow 0
\]
shows that any line bundle on $Z_k$ can be identified with the pullback $\pi^* \mathcal O_{\mathbb P^1} (n)$ of the projection $\pi \colon Z_k \to \mathbb P^1$ to the zero section of $Z_k $. We write $\mathcal O_{Z_k} (n)$ or $\mathcal O (n)$ for $\pi^* \mathcal O_{\mathbb P^1} (n)$, whose change of coordinates from $U$ to $V$ can be given by the transition function $z^{-n}$. Here $n \in \mathbb Z$ is the first Chern class of $\mathcal O (n)$.

To calculate Poisson structures, note that $\Wedge^2 \mathcal T_{Z_k}$ is the anticanonical line bundle. The transition matrix for the tangent bundle of $Z_k$ is given in canonical coordinates by the Jacobian matrix of the change of coordinates $(z, u) \mapsto (z^{-1}, z^k u)$ of the manifold. For $Z_k$
\begin{align}
\label{jacobian}
\operatorname{Jac}_{UV} =
\begin{pmatrix}
\frac{\partial}{\partial z} z^{-1} & \frac{\partial}{\partial u} z^{-1} \\[.75ex]
\frac{\partial}{\partial z} z^k u  & \frac{\partial}{\partial u} z^k u
\end{pmatrix}
=
\begin{pmatrix}
 -z^{-2} & 0 \\
k \, z^{k-1} u & z^k
\end{pmatrix}.
\end{align}
The transition function for the anticanonical line bundle is then given by the determinant of this Jacobian, which is $-z^{k-2}$. We can thus identify $\H^0 (Z_k, \Wedge^2 \mathcal T_{Z_k}) \simeq \H^0 (Z_k, \mathcal O (-k{+}2))$ as the space of Poisson structures.

\begin{lemma}
\label{poissonstructures}
A general Poisson structure $\sigma \in \H^0 (Z_k, \Wedge^2 \mathcal T_{Z_k})$ on $Z_k$ is given in canonical coordinates by
\[
\Big( \sigma_U \, \tfrac{\partial}{\partial z} {\wedge} \tfrac{\partial}{\partial u}, \sigma_V \, \tfrac{\partial}{\partial \xi} {\wedge} \tfrac{\partial}{\partial v} \Big)
\]
where $(\sigma_U, \sigma_V)$ is of the form
\begin{alignat*}{5}
(1) \quad&&                       (f_U + z \, g_U,& -\xi \, f_V - g_V) && \text{for $k = 1$} \\
(2) \quad&&                                  (f_U,& -f_V) && \text{for $k = 2$} \\
(3) \quad&& (u \, f_U + z u \, g_U + z^2 u \, h_U,& -\xi^2 v \, f_V - \xi v \, g_V - v \, h_V) \quad && \text{for $k \geq 3$}
\end{alignat*}
for any global functions $(f_U, f_V), (g_U, g_V), (h_U, h_V) \in \H^0 (Z_k, \mathcal O_{Z_k})$.

In the basis $\big( \tfrac{\partial}{\partial z} {\wedge} \tfrac{\partial}{\partial u}, \tfrac{\partial}{\partial \xi} {\wedge} \tfrac{\partial}{\partial v} \big)$, the space of Poisson structures $\H^0 (Z_k, \Wedge^2 \mathcal T_{Z_k})$ is thus generated by
\begin{alignat*}{5}
(1) \quad& (1, -\xi), (z, -1)                              && \text{for $k = 1$} \\
(2) \quad& (1, -1)                                         && \text{for $k = 2$} \\
(3) \quad& (u, -\xi^2 v), (z u, -\xi v), (z^2 u, -v) \quad && \text{for $k \geq 3$}
\end{alignat*}
as a module over global functions.
\end{lemma}

\begin{proof}
For (1) we need to calculate $\H^0 (Z_1, \mathcal O (1))$, which is Lem.~\ref{j>0}. For (2) and (3) we need to calculate $\H^0 (Z_k, \mathcal O (-k{+}2))$, which is Lem.~\ref{j<0}.
\end{proof}

\begin{remark}
The space of Poisson structures on $Z_k$ is of infinite dimension over $\mathbb C$, but restricted to the $n$th infinitesimal neighbourhood $\ell^{(n)}$, the space of Poisson structures is of dimension
\[
\begin{cases}
\tfrac{(n+1)(n+4)}2 & \text{for } k = 1 \\
n^2 & \text{for } k = 2 \\
\tfrac{n((n-1)k+4)}2 & \text{for } k \geq 3.
\end{cases}
\]
\end{remark}

\begin{notation}
Write $Z_{\geq m}$ for $Z_k$ with $k \geq m$.
\end{notation}

To describe explicit quantizations of Poisson structures on non-affine varieties, the following property will be useful.

\begin{definition}
\label{tangent}
A bivector field $\sigma \in \H^0 (X, \Wedge^2 \mathcal T_X)$ is {\it tangent} to a divisor $D \subset X$ if
\[
\{ \mathcal O_X, \mathcal I_D \} \subset \mathcal I_D
\]
{\it i.e.}\ if the ideal sheaf $\mathcal I_D$ of $D$ is a Poisson ideal. Geometrically, $\sigma$ is tangent to $D$ if for every function $f \in \mathcal O_X$, the restriction of the Hamiltonian vector field $X_f = \{ f{,} \blank \}_\sigma$ to the divisor $D$ is tangent to $D$.
\end{definition}

In \S\ref{quantizableimmersions} we shall quantize Poisson structures tangent to the complement of an affine coordinate chart, so we record their particular form.

\begin{proposition}
\label{tangentpoisson}
Consider the open immersion $\mathbb C^2 \simeq U \subset Z_k$ with complementary divisor $D = Z_k \setminus U = \{ (0, v) \in V \} \simeq \mathbb C$. Then the space of Poisson structures tangent to $D$ is generated by
\begin{alignat*}{5}
(1) \quad& (1, -\xi)                          && \text{for $k = 1$} \\
(2) \quad& (u, -\xi^2 v), (z u, -\xi v) \quad && \text{for $k \geq 2$}
\end{alignat*}
as a module over global functions.
\end{proposition}

\begin{proof}
This follows from Lem.~\ref{poissonstructures} and the observation that since $\mathcal I_D \vert_V = (\xi)$, the coefficient function $\sigma_V$ of a Poisson structure $(\sigma_U, \sigma_V)$ which is tangent to $D$ should be a multiple of $\xi$ in $V$-coordinates.
\end{proof}

Poisson structures on $Z_k$ are depicted in Fig.~\ref{monomialspoissonzk}, where dots represent the monomials (in $U$-coordinates) which appear in the expression of a general Poisson structure on $Z_k$ and circled dots those of a Poisson structure tangent to $D = Z_k \setminus U$.

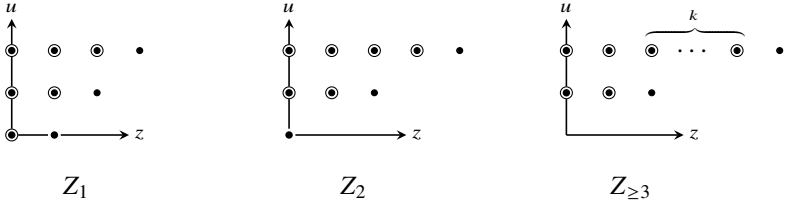
\begin{figure}
\centering
\begin{tikzpicture}[x=1.6em,y=1.6em]
\begin{scope}
\node at (1.5,-1.25) {$Z_1$};
\draw[-stealth,line width=.6pt] (0,0) -- (0,2.75);
\draw[-stealth,line width=.6pt] (0,0) -- (2.75,0);
\node[font=\scriptsize] at (0,3) {$u$};
\node[font=\scriptsize] at (3,0) {$z$};
\draw[color=white,fill=black,line width=1pt] (0,0) circle(1.9pt);
\draw[color=white,fill=black,line width=1pt] (1,0) circle(1.9pt);
\draw[color=white,fill=black,line width=1pt] (0,1) circle(1.9pt);
\draw[fill=black] (1,1) circle(.25ex);
\draw[fill=black] (2,1) circle(.25ex);
\draw[color=white,fill=black,line width=1pt] (0,2) circle(1.9pt);
\draw[fill=black] (1,2) circle(.25ex);
\draw[fill=black] (2,2) circle(.25ex);
\draw[fill=black] (3,2) circle(.25ex);
\draw[fill=none] (0,0) circle(.5ex);
\draw[fill=none] (0,1) circle(.5ex);
\draw[fill=none] (1,1) circle(.5ex);
\draw[fill=none] (0,2) circle(.5ex);
\draw[fill=none] (1,2) circle(.5ex);
\draw[fill=none] (2,2) circle(.5ex);
\end{scope}
\begin{scope}[shift={(6.5,0)}]
\node at (1.5,-1.25) {$Z_2$};
\draw[-stealth,line width=.6pt] (0,0) -- (0,2.75);
\draw[-stealth,line width=.6pt] (0,0) -- (2.75,0);
\node[font=\scriptsize] at (0,3) {$u$};
\node[font=\scriptsize] at (3,0) {$z$};
\draw[color=white,fill=black,line width=1pt] (0,0) circle(1.9pt);
\draw[color=white,fill=black,line width=1pt] (0,1) circle(1.9pt);
\draw[fill=black] (1,1) circle(.25ex);
\draw[fill=black] (2,1) circle(.25ex);
\draw[color=white,fill=black,line width=1pt] (0,2) circle(1.9pt);
\draw[fill=black] (1,2) circle(.25ex);
\draw[fill=black] (2,2) circle(.25ex);
\draw[fill=black] (3,2) circle(.25ex);
\draw[fill=black] (4,2) circle(.25ex);
\draw[fill=none] (0,1) circle(.5ex);
\draw[fill=none] (1,1) circle(.5ex);
\draw[fill=none] (0,2) circle(.5ex);
\draw[fill=none] (1,2) circle(.5ex);
\draw[fill=none] (2,2) circle(.5ex);
\draw[fill=none] (3,2) circle(.5ex);
\end{scope}
\begin{scope}[shift={(13,0)}]
\node at (1.5,-1.25) {$Z_{\geq 3}$};
\draw[-stealth,line width=.6pt] (0,0) -- (0,2.75);
\draw[-stealth,line width=.6pt] (0,0) -- (2.75,0);
\node[font=\scriptsize] at (0,3) {$u$};
\node[font=\scriptsize] at (3,0) {$z$};
\draw[color=white,fill=black,line width=1pt] (0,1) circle(1.9pt);
\draw[color=white,fill=black,line width=1pt] (0,2) circle(1.9pt);
\draw[fill=black] (1,1) circle(.25ex);
\draw[fill=black] (2,1) circle(.25ex);
\draw[fill=black] (1,2) circle(.25ex);
\draw[fill=black] (2,2) circle(.25ex);
\draw (3,2) node {$\dotsc$};
\draw[fill=black] (4,2) circle(.25ex);
\draw[fill=black] (5,2) circle(.25ex);
\draw[fill=none] (0,1) circle(.5ex);
\draw[fill=none] (1,1) circle(.5ex);
\draw[fill=none] (0,2) circle(.5ex);
\draw[fill=none] (1,2) circle(.5ex);
\draw[fill=none] (2,2) circle(.5ex);
\draw[fill=none] (4,2) circle(.5ex);
\draw (3,2.6) node[font=\scriptsize] {$\overbrace{\hspace{5.25em}}^{k}$};
\end{scope}
\end{tikzpicture}
\caption{Monomials of Poisson structures on $Z_k$}
\label{monomialspoissonzk}
\end{figure}

Recall that the {\it $r$th degeneracy locus} of a holomorphic Poisson structure on a complex manifold or algebraic variety $X$ is defined as
\[
D_{2r} (\sigma) = \{ x \in X \mid \rank \sigma (x) \leq 2r \},
\]
where $\sigma$ is viewed as a map $\mathcal T_X^* \to \mathcal T_X$ by contracting a $1$-form with the bivector field $\sigma$. At a given point on a complex {\it surface} a holomorphic Poisson structure has either full rank, or rank $0$. Thus, for the surfaces $Z_k$ we call $\Dgn (\sigma) := D_0 (\sigma)$ {\it the degeneracy locus} of $\sigma$ and this degeneracy locus is given by the zeros of the coefficient functions $\sigma_U$ and $\sigma_V$.

A non-degenerate holomorphic Poisson structure $\sigma$ is called a {\it holomorphic symplectic} structure as $\sigma$ determines a non-degenerate closed holomorphic $2$-form $\omega$ by
\[
\omega (X_f, X_g) = \{ f, g \}_\sigma
\]
where $X_f$ denotes the Hamiltonian vector field associated to a function $f$.

As we shall see in Prop.~\ref{holomorphicsymplectic}, $Z_2$ is the only one of the surfaces $Z_k$ that admits a holomorphic symplectic structure and when this structure is algebraic, it is unique up to scaling.

\begin{example}
\label{degeneracyexample}
If $\sigma$ is as in Lem.~\ref{poissonstructures} with global functions $(f_U, f_V)$, $(g_U, g_V)$, $(h_U, h_V)$ simply (nonzero) {\it constants}, we have
\[
\Dgn (\sigma)
= 
\begin{cases}
\pi^{-1} (x) & \text{for $k=1$} \\
\emptyset & \text{for $k=2$} \\
\pi^{-1} (x) \cup \pi^{-1} (y) \cup \ell & \text{for $k \geq 3$}
\end{cases}
\]
where $\ell \subset Z_k$ is the zero section and $x, y \in \mathbb P^1$ are two points, possibly equal. These degeneracy loci are depicted in Fig.~\ref{degeneracyexamplefigure}.

For $k = 2$ the Poisson structure is non-degenerate and defines a holomorphic symplectic form.
\end{example}

If $\sigma$ is a Poisson structure which is tangent to a fibre $D = \pi^{-1} (x) \simeq \mathbb C$ of the bundle projection $\pi \colon Z_k \to \mathbb P^1$, the picture of Fig.~\ref{degeneracyexamplefigure} reduces to two cases:
\begin{equation}
\label{degeneracytangent}
\Dgn (\sigma) =
\begin{cases}
\pi^{-1} (x) & \text{for $k=1$} \\
\pi^{-1} (x) \cup \ell & \text{for $k \geq 2$.} \\
\end{cases}
\end{equation}
We can set $U = Z_k \setminus D$ so that $x \in \mathbb P^1$ is given in canonical coordinates by $\xi = 0$ and in these coordinates the form of $\sigma$ is given in Prop.~\ref{tangentpoisson}.

Equation (\ref{degeneracytangent}) describes the degeneracy locus for a non-zero linear combination of generators of those Poisson structures tangent to a fibre $D = \pi^{-1} (x)$. The degeneracy locus of a general Poisson structure $\sigma$ tangent to $D$ may be more complicated. However, to construct quantizations of $\sigma$, it is only necessary to require that $\Dgn (\sigma)$ contain a fibre of the projection $\pi$. The effect of quantizations on moduli then only depends on whether $\Dgn (\sigma)$ also contains all of $\ell$ or not and so we introduce the following notation.

\begin{notation}
\label{minimally}
Write $\sigma_0$, respectively $\sigma$, for Poisson structures on $Z_k$ such that
\begin{align*}
\ell \,\not\subset \eqmakebox[dgn]{$\Dgn (\sigma_0)$} \supset\, \pi^{-1} (x) \\
\ell \,\subset \eqmakebox[dgn]{$\Dgn (\sigma)$} \supset\, \pi^{-1} (x)
\end{align*}
respectively.

Note that $\sigma_0$ is a ``minimally degenerate'' Poisson structure on $Z_1$, whereas $\sigma$ denotes Poisson structures which are degenerate on all of $\ell$, which can occur on $Z_k$ for all $k \geq 1$.
\end{notation}

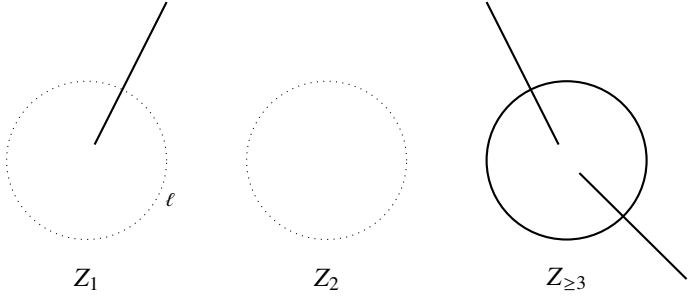
\begin{figure}
\begin{align*}
\begin{tikzpicture}[x=3em,y=3em]
\draw[dotted] (0,0) circle (1);
\path[thick] (0.1,0.2) edge (1,2);
\draw[dotted] (3,0) circle (1);
\draw[thick] (6,0) circle (1);
\path[thick] (5.9,0.2) edge (5,2);
\path[thick] (6.16,-0.16) edge (7.5,-1.5);
\node[font=\scriptsize] at (1.05,-0.5) {$\ell$};
\node at (0,-1.5) {$Z_1$};
\node at (3,-1.5) {$Z_2$};
\node at (6,-1.5) {$Z_{\geq 3}$};
\end{tikzpicture}
\end{align*}
\caption{The (real part of the) degeneracy loci of the Poisson structures of Ex.~\ref{degeneracyexample}.}
\label{degeneracyexamplefigure}
\end{figure}

\begin{proposition}
\label{holomorphicsymplectic}
$Z_k$ admits a holomorphic symplectic structure if and only if $k = 2$. Moreover, if the holomorphic symplectic structure on $Z_2$ is algebraic, then it is a constant multiple of the canonical symplectic structure on $Z_2 \simeq T^* \mathbb P^1$.
\end{proposition}

\begin{proof}
A holomorphic symplectic structure corresponds to a nowhere vanishing  global holomorphic section of $\Wedge^2 \mathcal T_{Z_k} \simeq \mathcal O (-k{+}2)$,  which exists only when this bundle is trivial, {\it i.e.}\ when $k = 2$.

For $Z_2$, the canonical non-degenerate Poisson structure is given by $(1, -1)$, so that in coordinates the bracket of two functions $f, g$ is given by
\begin{align*}
\{ f, g \}\vert_U &= \phantom{-} \big\langle \tfrac{\partial}{\partial z} \wedge \tfrac{\partial}{\partial u}, \mathrm d f_U \wedge \mathrm d g_U \big\rangle \\
\{ f, g \}\vert_V &= -\big\langle \tfrac{\partial}{\partial \xi} \wedge \tfrac{\partial}{\partial v}, \mathrm d f_V \wedge \mathrm d g_U \big\rangle.
\end{align*}
If a non-degenerate Poisson structure on $Z_2$ is algebraic, {\it i.e.}\ an algebraic section of $\Wedge^2 \mathcal T_{Z_2}$, it is a constant multiple of the canonical non-degenerate Poisson structure, as every non-constant polynomial in $z, u$ has at least one zero. (This does not hold in the analytic category, as there are non-constant non-vanishing complex analytic sections.)
\end{proof}

\begin{proof}[Alternative proof.]
We shall see in \S\ref{instantons} that $Z_k$ is the minimal resolution of the $\frac1k (1,1)$ surface singularity $X_k \simeq \mathbb C^2 / \Gamma$, where $\Gamma$ is generated by $\gamma = \left( \begin{smallmatrix} \omega & 0 \\ 0 & \omega \end{smallmatrix} \right)$ for $\omega$ a primitive $k$th root of unity. A resolution of $\mathbb C^2 / \Gamma$ admits a holomorphic symplectic form if and only if $\Gamma \subset \SL (2, \mathbb C)$. However, $\det \gamma = \omega^2 = 1$ if and only if $k = 2$. \end{proof}

\begin{remark} 
Here we discuss the case of {\it degenerate} Poisson structures on $Z_k$, leaving the holomorphic symplectic case on $Z_2$ for future work. More precisely, our construction works for Poisson structures tangent to a fibre of $\pi \colon Z_k \to \mathbb P^1$ (see Def.~\ref{tangent} and Prop.~\ref{extends}), which the holomorphic symplectic structure on $Z_2$ does not satisfy ({\it cf.}\ Rem.~\ref{doesnotextendz2}).
\end{remark}

\section{Deformation quantization and star products}
\label{quantization}

Let $\hbar$ be a formal variable and denote by $\mathcal O_X \llrr{\hbar}$ the completed tensor product $\mathcal O_X \hatotimes \mathbb C \llrr{\hbar}$, viewed as a sheaf of $\mathbb C \llrr{\hbar}$-vector spaces on $X$, where a ``section'' over an open set $U$ is given by a formal power series $f = \sum_{n = 0}^\infty f_n \hbar^n$ with each $f_n \in \mathcal O_X (U)$. We shall turn $\mathcal O_X \llrr{\hbar}$ into a sheaf of associative $\mathbb C \llrr{\hbar}$-algebras by formally deforming the usual commutative product on functions. The augmentation $\mathbb C \llrr{\hbar} \to \mathbb C$ induces an augmentation $\mathcal O_X \llrr{\hbar} \to \mathcal O_X$.

\begin{definition}
\label{defstarproduct}
A {\it star product} on a complex manifold (or smooth complex algebraic variety) $X$ is a $\mathbb C \llrr{\hbar}$-bilinear associative product
\begin{align*}
\star \colon \mathcal O_X \llrr{\hbar} \times \mathcal O_X \llrr{\hbar} \to \mathcal O_X \llrr{\hbar}
\end{align*}
which is of the form
\[
(f, g) \mapsto fg + \sum_{n=1}^\infty B_n (f,g) \, \hbar^n
\]
where the $B_n$ are bidifferential operators, {\it i.e.}\ bilinear operators which are differential operators in both arguments. 
\end{definition}

\begin{definition}\label{ge}
\cite[\S 0.1]{yekutieli}
Two star products $\star, \star'$ on $X$ are said to be {\it gauge equivalent}, if there exists an isomorphism $(\mathcal O_X \llrr{\hbar}, \star) \simeq (\mathcal O_X \llrr{\hbar}, \star')$ which commutes with the augmentations $\mathcal O_X \llrr{\hbar} \to \mathcal O_X$.
\end{definition}

\begin{definition}
\label{dq}
Let $(X, \sigma)$ be a holomorphic Poisson manifold with associated Poisson bracket $\{ \blank {,} \blank \}_\sigma$. A {\it deformation quantization} of $(X, \sigma)$ is a pair $(X, \star_\sigma)$, where $\star_\sigma$ is a star product on $X$ with $B_1 (f, g) = \{ f, g \}_\sigma$, that is 
\[
f \star_\sigma g = f g + \{ f, g \}_\sigma \hbar + \dotsb
\]
We set $\mathcal A^\sigma := (\mathcal O \llrr{\hbar},\star_\sigma)$ the sheaf of formal functions with holomorphic coefficients on the quantization $(X, \star_\sigma)$.

We call $\mathcal Z_k (\sigma) = (Z_k, \mathcal A^\sigma)$ a {\it noncommutative deformation} of $Z_k$. As we usually work with a specified fixed Poisson structure, we usually use the abbreviated notations $\mathcal A$, $\{ \blank {,} \blank \}$ and $\star$.
\end{definition}

The existence of star products was first proved in the $C^\infty$ setting in Kontsevich's seminal paper \cite{kontsevich1} and later generalized to the algebro-geometric setting \cite{kontsevich2,yekutieli}. We quote this generalization in a simplified form.

\begin{theorem}
\label{twisteddeformationquantization}
\cite[Cor.\ 11.2]{yekutieli}
Let $X$ be a complex algebraic variety with structure sheaf $\mathcal O_X$ and assume that $\H^1 (X, \mathcal O_X)$ and $\H^2 (X, \mathcal O_X)$ vanish. Then there is a bijection
\[
\{ \text{\rm Poisson deformations of $\mathcal O_X$} \} / {\sim} \; \leftrightarrow \; \{ \text{\rm associative deformations of $\mathcal O_X$} \} / {\sim}
\]
where $\sim$ denotes gauge equivalence.
\end{theorem}

Note that the surfaces $Z_k$ satisfy the hypothesis of Thm.~\ref{twisteddeformationquantization}.

\begin{remark}
Working with the sheaf of algebras $(\mathcal O_X \llrr{\hbar}, \star)$ we are restricting ourselves to deformations which are in some sense ``purely noncommutative''. More generally \cite{kontsevich2,kashiwaraschapira,yekutieli}, in the non-affine setting one may consider formal deformations of $\mathcal O_X$, which do not necessarily have $\mathcal O_X \llrr{\hbar}$ as underlying sheaf of $\mathbb C \llrr{\hbar}$-vector spaces, but could simultaneously deform $\mathcal O_X$ in some commutative direction, corresponding to simultaneously deforming the restriction morphisms of the sheaf $\mathcal O_X$. Although $Z_k$ does admit commutative deformations, we are restricting ourselves to these purely noncommutative directions for two reasons. On the one hand, we expect that turning on a commutative direction of deformation would imply that the moduli spaces of vector bundles become very small or even trivial, as is the case for the purely commutative deformations \cite{BrG}. On the other hand, it seems difficult to obtain explicit simultaneous deformations in the generality we achieve in this article, as it was shown in \cite[\S 5.3]{barmeierfregier} that simultaneous commutative and noncommutative deformations of $Z_k$ may be obstructed.
\end{remark}

\subsection{The Kontsevich star product on $\mathbb C^d$}
\label{kontsevichstarproduct}

As part of an explicit quasi-isomorphism of L$_\infty$ algebras in the proof of his formality theorem, Kontsevich \cite{kontsevich1} gave an explicit construction for a star product on $(\mathbb R^d, \pi)$ for some (real) Poisson structure $\pi$. The formula applies without change to the complex setting for a (holomorphic) Poisson structure on $\mathbb C^d$, which we will use in \S \ref{quantizableimmersions} to give explicit star products on $Z_k$.

We briefly recall the definition of this star product and refer to \cite{kontsevich1,kontsevich2} for details.

\begin{definition}
Let $\sigma$ be a Poisson structure on $\mathbb C^d$. The {\it Kontsevich star product} $\star^{\mathrm K}_\sigma$ on $\mathbb C^d$ is given by
\begin{equation}
\label{kontsevichstar}
f \star^{\mathrm K}_\sigma g = f g + \sum_{n = 1}^\infty \hbar^n \sum_{\Gamma \in \mathfrak G_{n,2}} w_\Gamma \, B_\Gamma (f, g)
\end{equation}
where
\begin{itemize}
\item $\mathfrak G_{n,2}$ is the set of {\it admissible} graphs with $n$ unfilled (``first type'') and $2$ filled (``second type'') vertices,
\item $B_\Gamma$ is the bidifferential operator for $\sigma$ associated to the graph $\Gamma$, and
\item $w_\Gamma$ is the {\it weight} of the graph $\Gamma$ obtained as the integral over a certain configuration space.
\end{itemize}
\end{definition}

An admissible graph $\Gamma \in \mathfrak G_{n,2}$ has two filled vertices representing the two entries of $B_\Gamma$ and $n$ unfilled vertices. Arrows which start at unfilled vertices represent derivatives of the target of the arrow. We shall use the notation $\partial_i$ for the derivative with respect to the $i$th coordinate of $\mathbb C^d$ and write $\sigma^{ij}$ for the coefficient function of $\partial_i \wedge \partial_j$. The $\sigma^{ij}$ are holomorphic functions on $\mathbb C^d$ and define a skew-symmetric $d \times d$ matrix.

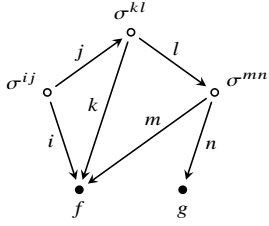
\begin{figure}
\begin{tikzpicture}[x=3.3em,y=3.3em]
\draw[line width=.6pt, fill=white]         (18:1) circle(0.3ex);
\node[shape=circle, scale=0.8](right)  at  (18:1) {};
\node[font=\scriptsize]                at  (18:1.4) {$\sigma^{mn}$\strut};
\draw[line width=.6pt, fill=white]         (90:1) circle(0.3ex);
\node[shape=circle, scale=0.8](middle) at  (90:1) {};
\node[font=\scriptsize]                at  (90:1.25) {$\sigma^{kl}$};
\draw[line width=.6pt, fill=white]        (162:1) circle(0.3ex);
\node[shape=circle, scale=0.8](left)   at (162:1) {};
\node[font=\scriptsize]                at (162:1.3) {$\sigma^{ij}$\strut};
\draw[line width=.6pt, fill=black]        (234:1) circle(0.3ex);
\node[shape=circle, scale=0.8](f)      at (234:1) {};
\draw[line width=.6pt, fill=black]        (306:1) circle(0.3ex);
\node[shape=circle, scale=0.8](g)      at (306:1) {};
\node[font=\scriptsize,below]     at (f) {$f$\strut};
\node[font=\scriptsize,below]     at (g) {$g$\strut};
\path[-stealth, line width=.6pt, line cap=round] (left) edge
node[font=\scriptsize,pos=.5,left=-.1ex] {$i$} (f);
\path[-stealth, line width=.6pt, line cap=round] (left) edge
node[font=\scriptsize,pos=.37,above=-.15ex] {$j$} (middle);
\path[-stealth, line width=.6pt, line cap=round] (middle) edge
node[font=\scriptsize,pos=.45,left=-.1ex] {$k$} (f);
\path[-stealth, line width=.6pt, line cap=round] (middle) edge 
node[font=\scriptsize,pos=.55,above] {$l$} (right);
\path[-stealth, line width=.6pt, line cap=round] (right) edge 
node[font=\scriptsize,pos=.45,above=.2ex] {$m$} (f);
\path[-stealth, line width=.6pt, line cap=round] (right) edge 
node[font=\scriptsize,pos=.55,right=-.15ex] {$n$} (g);
\end{tikzpicture}
\caption{The graph $\Gamma$ of Ex.~\ref{examplegraph}.}
\label{figureexamplegraph}
\end{figure}

\begin{example}
\label{examplegraph}
The bidifferential operator for the graph $\Gamma$ of Fig.~\ref{figureexamplegraph} is given by
\[
B_\Gamma (f, g) = \sum_{1 \leq i,j,k,l,m,n \leq d} \sigma^{ij} \, \partial_j (\sigma^{kl}) \, \partial_l (\sigma^{mn}) \, \partial_i \partial_k \partial_m (f) \, \partial_n (g).
\]
\end{example}

The graphs with arrows ending in unfilled vertices represent bidifferential operators which take derivatives of $\sigma^{ij}$. 
When the Poisson structure $\sigma$ is constant and non-degenerate, {\it i.e.}\
\begin{flalign*}
&& \sigma = \sum_{i,j} \sigma^{ij} \; \partial_i \wedge \partial_j && \mathllap{\sigma^{ij}=-\sigma^{ji} \in \mathbb C}
\end{flalign*}
graphs with arrows ending in unfilled vertices represent the zero operator and thus, for each $n$, there is only one graph in $\mathfrak G_{n,2}$ which contributes to $B_n$ (see Fig.~\ref{moyalgraphs}). For $\sigma$ constant, the Kontsevich star product $\star^{\mathrm K}_\sigma$ then coincides with the {\it Moyal product} $\star^{\mathrm M}_\sigma$ given by
\begin{align}
\label{moyal}
f \star^{\mathrm M}_\sigma g &= fg + \hbar \displaystyle\sum_{i,j} \sigma^{ij} \; \partial_i(f) \; \partial_j(g) + \frac{\hbar^2}{2} \sum_{i,j,k,l} \sigma^{ij} \sigma^{kl} \; \partial_i\partial_k(f) \; \partial_j \partial_l (g) + \dotsb \notag\\
&= \sum_{n=0}^{\infty} \frac{\hbar^n}{n!} \sum_{i_1,\dotsc,i_n,j_1,\dotsc,j_n} \left( \prod_{k=1}^n \sigma^{i_k j_k} \right) \times \left( \prod_{k=1}^n \partial_{i_k} \right) (f) \times\left( \prod_{k=1}^n \partial_{j_k} \right) (g),
\end{align}
where the symbol $\times$ denotes the usual product.

\begin{figure}
\begin{tikzpicture}[x=3em,y=3em]
\draw[line width=.6pt, fill=white] (.5,.886) circle(0.3ex);
\node[shape=circle, scale=0.8](one) at (.5,.886) {};
\draw[line width=.6pt, fill=black] (1,0) circle(0.3ex);
\node[shape=circle, scale=0.8](f) at (1,0) {};
\draw[line width=.6pt, fill=black] (0,0) circle(0.3ex);
\node[shape=circle, scale=0.8](g) at (0,0) {};
\path[-stealth, line width=.6pt, line cap=round] (one) edge (f);
\path[-stealth, line width=.6pt, line cap=round] (one) edge (g);
\end{tikzpicture}
\hspace{2.5em}
\begin{tikzpicture}[x=3em,y=3em]
\draw[line width=.6pt, fill=white] (0,1) circle(0.3ex);
\node[shape=circle, scale=0.8](one) at (0,1) {};
\draw[line width=.6pt, fill=white] (1,1) circle(0.3ex);
\node[shape=circle, scale=0.8](two) at (1,1) {};
\draw[line width=.6pt, fill=black] (1,0) circle(0.3ex);
\node[shape=circle, scale=0.8](f) at (1,0) {};
\draw[line width=.6pt, fill=black] (0,0) circle(0.3ex);
\node[shape=circle, scale=0.8](g) at (0,0) {};
\path[-stealth, line width=.6pt, line cap=round] (one) edge (f);
\path[-stealth, line width=.6pt, line cap=round] (one) edge (g);
\path[-stealth, line width=.6pt, line cap=round] (two) edge (f);
\path[-stealth, line width=.6pt, line cap=round] (two) edge (g);
\end{tikzpicture}
\hspace{2.5em}
\begin{tikzpicture}[x=3em,y=3em]
\draw[line width=.6pt, fill=white] (162:.851) circle(0.3ex);
\node[shape=circle, scale=0.8](one) at (162:.851) {};
\draw[line width=.6pt, fill=white] (90:.851) circle(0.3ex);
\node[shape=circle, scale=0.8](two) at (90:.851) {};
\draw[line width=.6pt, fill=white] (18:.851) circle(0.3ex);
\node[shape=circle, scale=0.8](three) at (18:.851) {};
\draw[line width=.6pt, fill=black] (234:.851) circle(0.3ex);
\node[shape=circle, scale=0.8](f) at (234:.851) {};
\draw[line width=.6pt, fill=black] (306:.851) circle(0.3ex);
\node[shape=circle, scale=0.8](g) at (306:.851) {};
\path[-stealth, line width=.6pt, line cap=round] (one) edge (f);
\path[-stealth, line width=.6pt, line cap=round] (one) edge (g);
\path[-stealth, line width=.6pt, line cap=round] (two) edge (f);
\path[-stealth, line width=.6pt, line cap=round] (two) edge (g);
\path[-stealth, line width=.6pt, line cap=round] (three) edge (f);
\path[-stealth, line width=.6pt, line cap=round] (three) edge (g);
\end{tikzpicture}
\hspace{2.5em}
\begin{tikzpicture}[x=3em,y=3em]
\draw[line width=.6pt, fill=white] (60:1) circle(0.3ex);
\node[shape=circle, scale=0.8](one) at (60:1) {};
\draw[line width=.6pt, fill=white] (120:1) circle(0.3ex);
\node[shape=circle, scale=0.8](two) at (120:1) {};
\draw[line width=.6pt, fill=white] (180:1) circle(0.3ex);
\node[shape=circle, scale=0.8](three) at (180:1) {};
\draw[line width=.6pt, fill=white] (0:1) circle(0.3ex);
\node[shape=circle, scale=0.8](four) at (0:1) {};
\draw[line width=.6pt, fill=black] (240:1) circle(0.3ex);
\node[shape=circle, scale=0.8](f) at (240:1) {};
\draw[line width=.6pt, fill=black] (300:1) circle(0.3ex);
\node[shape=circle, scale=0.8](g) at (300:1) {};
\path[-stealth, line width=.6pt, line cap=round] (one) edge (f);
\path[-stealth, line width=.6pt, line cap=round] (one) edge (g);
\path[-stealth, line width=.6pt, line cap=round] (two) edge (f);
\path[-stealth, line width=.6pt, line cap=round] (two) edge (g);
\path[-stealth, line width=.6pt, line cap=round] (three) edge (f);
\path[-stealth, line width=.6pt, line cap=round] (three) edge (g);
\path[-stealth, line width=.6pt, line cap=round] (four) edge (f);
\path[-stealth, line width=.6pt, line cap=round] (four) edge (g);
\end{tikzpicture}
\caption{The graphs in $\mathfrak G_{n,2}$ for $n = 1, 2, 3, 4$ contributing to $B_n$ in the Moyal product.}
\label{moyalgraphs}
\end{figure}
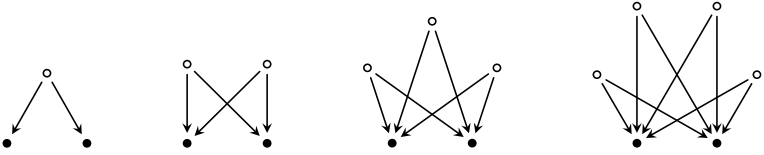

In general a holomorphic Poisson structure $\sigma$ will not be constant, not even locally. (For the surfaces $Z_k$ see Lem.~\ref{poissonstructures} and Prop.~\ref{holomorphicsymplectic}.) In this case the formula for the Kontsevich star product also involves derivatives of the coefficient functions $\sigma^{ij}$.

\begin{lemma}
\label{secondorder}
\cite{dito}
Up to second order in $\hbar$ the Kontsevich star product for $\sigma$ on $\mathbb C^d$ is given by
\begin{flalign*}
&& f \star^{\mathrm K}_\sigma g ={}& f g \\
&& &+ \hbar \displaystyle\sum_{i,j} \sigma^{ij} \; \partial_i(f) \; \partial_j(g) &
\mathllap{\begin{tikzpicture}[x=2em,y=2em,baseline=.7em]
\draw[line width=.6pt, fill=black] (1,0) circle(0.25ex);
\draw[line width=.6pt, fill=black] (0,0) circle(0.25ex);
\draw[line width=.6pt, fill=white] (.5,.886) circle(0.25ex);
\node[shape=circle, scale=0.8](f) at (0,0) {};
\node[shape=circle, scale=0.8](g) at (1,0) {};
\node[shape=circle, scale=0.8](one) at (.5,.886) {};
\path[-stealth, line width=.6pt, line cap=round] (one) edge (f);
\path[-stealth, line width=.6pt, line cap=round] (one) edge (g);
\end{tikzpicture}}
\\
&& &+ \frac{\hbar^2}2 \sum_{i,j,k,l} \sigma^{ij} \sigma^{kl} \; \partial_i \partial_k (f) \; \partial_j \partial_l (g) &
\mathllap{\begin{tikzpicture}[x=2em,y=2em,baseline=.7em]
\draw[line width=.6pt, fill=black] (1,0) circle(0.25ex);
\draw[line width=.6pt, fill=black] (0,0) circle(0.25ex);
\draw[line width=.6pt, fill=white] (0,1) circle(0.25ex);
\draw[line width=.6pt, fill=white] (1,1) circle(0.25ex);
\node[shape=circle, scale=0.7](f) at (0,0) {};
\node[shape=circle, scale=0.7](g) at (1,0) {};
\node[shape=circle, scale=0.7](L) at (0,1) {};
\node[shape=circle, scale=0.7](R) at (1,1) {};
\path[-stealth, line width=.6pt, line cap=round] (L) edge (f);
\path[-stealth, line width=.6pt, line cap=round] (L) edge (g);
\path[-stealth, line width=.6pt, line cap=round] (R) edge (f);
\path[-stealth, line width=.6pt, line cap=round] (R) edge (g);
\end{tikzpicture}}
\\
&& &+ \frac{\hbar^2}3 \sum_{i,j,k,l} \sigma^{ij} \; \partial_i (\sigma^{kl}) \; \partial_j \partial_l (f) \; \partial_k (g) &
\mathllap{\begin{tikzpicture}[x=2em,y=2em,baseline=.7em]
\draw[line width=.6pt, fill=black] (1,0) circle(0.25ex);
\draw[line width=.6pt, fill=black] (0,0) circle(0.25ex);
\draw[line width=.6pt, fill=white] (0,1) circle(0.25ex);
\draw[line width=.6pt, fill=white] (1,1) circle(0.25ex);
\node[shape=circle, scale=0.7](f) at (0,0) {};
\node[shape=circle, scale=0.7](g) at (1,0) {};
\node[shape=circle, scale=0.7](L) at (0,1) {};
\node[shape=circle, scale=0.7](R) at (1,1) {};
\path[-stealth, line width=.6pt, line cap=round] (L) edge (f);
\path[-stealth, line width=.6pt, line cap=round] (L) edge (R);
\path[-stealth, line width=.6pt, line cap=round] (R) edge (f);
\path[-stealth, line width=.6pt, line cap=round] (R) edge (g);
\end{tikzpicture}}
\\
&& &+ \frac{\hbar^2}3 \sum_{i,j,k,l} \sigma^{kl} \; \partial_k (\sigma^{ij}) \; \partial_i (f) \; \partial_j \partial_l (g) &
\mathllap{\begin{tikzpicture}[x=2em,y=2em,baseline=.7em]
\draw[line width=.6pt, fill=black] (1,0) circle(0.25ex);
\draw[line width=.6pt, fill=black] (0,0) circle(0.25ex);
\draw[line width=.6pt, fill=white] (0,1) circle(0.25ex);
\draw[line width=.6pt, fill=white] (1,1) circle(0.25ex);
\node[shape=circle, scale=0.7](f) at (0,0) {};
\node[shape=circle, scale=0.7](g) at (1,0) {};
\node[shape=circle, scale=0.7](L) at (0,1) {};
\node[shape=circle, scale=0.7](R) at (1,1) {};
\path[-stealth, line width=.6pt, line cap=round] (R) edge (g);
\path[-stealth, line width=.6pt, line cap=round] (R) edge (L);
\path[-stealth, line width=.6pt, line cap=round] (L) edge (g);
\path[-stealth, line width=.6pt, line cap=round] (L) edge (f);
\end{tikzpicture}}
\\
&& &- \frac{\hbar^2}6 \sum_{i,j,k,l} \partial_l (\sigma^{ij}) \; \partial_j (\sigma^{kl}) \; \partial_i (f) \; \partial_k (g) &
\mathllap{\begin{tikzpicture}[x=2em,y=2em,baseline=.7em]
\draw[line width=.6pt, fill=black] (1,0) circle(0.25ex);
\draw[line width=.6pt, fill=black] (0,0) circle(0.25ex);
\draw[line width=.6pt, fill=white] (0,1) circle(0.25ex);
\draw[line width=.6pt, fill=white] (1,1) circle(0.25ex);
\node[shape=circle, scale=0.7](f) at (0,0) {};
\node[shape=circle, scale=0.7](g) at (1,0) {};
\node[shape=circle, scale=0.7](L) at (0,1) {};
\node[shape=circle, scale=0.7](R) at (1,1) {};
\path[-stealth, line width=.6pt, line cap=round] (L) edge[out=15,   in=165] (R);
\path[-stealth, line width=.6pt, line cap=round] (L) edge (f);
\path[-stealth, line width=.6pt, line cap=round] (R) edge[out=-165, in=-15] (L);
\path[-stealth, line width=.6pt, line cap=round] (R) edge (g);
\end{tikzpicture}}
\\
&& &+ \dotsb
\end{flalign*}
\end{lemma}

To give the Kontsevich star product explicitly to higher orders one would have to compute the weights of all admissible graphs. Already for a single graph this computation is non-trivial, yet necessary even if one were restrict oneself to, say, linear Poisson structures. For example, Felder--Willwacher \cite{felderwillwacher} computed the weight of the graph shown in Fig.~\ref{felderwillwachergraph} to be $\zeta(3)^2 \big/ \pi^6$ up to rationals, omitting 50 terms known to be rational. (The precise weight was recently given as $\frac{13}{2903040} + \frac{\zeta (3)^2}{256 \pi^6}$ in \cite{BPP}, which relates the weights to multiple zeta values with integer coefficients.) As each unfilled vertex in this graph has only one ingoing arrow, the associated bidifferential operator is non-zero even for linear Poisson structures, and thus it will contribute to the expression of $B_7$ for any non-constant Poisson structure.

\begin{figure}
\begin{tikzpicture}[x=3em,y=3em]
\draw[line width=.6pt, fill=white] (330:1.467) circle(0.3ex);
\node[shape=circle, scale=0.8](seven) at (330:1.467) {};
\draw[line width=.6pt, fill=white] (10:1.467) circle(0.3ex);
\node[shape=circle, scale=0.8](six) at (10:1.467) {};
\draw[line width=.6pt, fill=white] (50:1.467) circle(0.3ex);
\node[shape=circle, scale=0.8](five) at (50:1.467) {};
\draw[line width=.6pt, fill=white] (90:1.467) circle(0.3ex);
\node[shape=circle, scale=0.8](four) at (90:1.467) {};
\draw[line width=.6pt, fill=white] (130:1.467) circle(0.3ex);
\node[shape=circle, scale=0.8](three) at (130:1.467) {};
\draw[line width=.6pt, fill=white] (170:1.467) circle(0.3ex);
\node[shape=circle, scale=0.8](two) at (170:1.467) {};
\draw[line width=.6pt, fill=white] (210:1.467) circle(0.3ex);
\node[shape=circle, scale=0.8](one) at (210:1.467) {};
\draw[line width=.6pt, fill=black] (250:1.467) circle(0.3ex);
\node[shape=circle, scale=1](f) at (250:1.467) {};
\draw[line width=.6pt, fill=black] (290:1.467) circle(0.3ex);
\node[shape=circle, scale=1](g) at (290:1.467) {};
\path[-stealth, line width=.6pt, line cap=round] (one) edge (f.165);
\path[-stealth, line width=.6pt, line cap=round] (one) edge (two.270);
\path[-stealth, line width=.6pt, line cap=round] (two.315) edge (f.130);
\path[-stealth, line width=.6pt, line cap=round] (two) edge (three);
\path[-stealth, line width=.6pt, line cap=round] (three) edge (f.95);
\path[-stealth, line width=.6pt, line cap=round] (three) edge (four);
\path[-stealth, line width=.6pt, line cap=round] (four) edge (f.60);
\path[-stealth, line width=.6pt, line cap=round] (four) edge (five);
\path[-stealth, line width=.6pt, line cap=round] (five) edge (six);
\path[-stealth, line width=.6pt, line cap=round] (five) edge (g.90);
\path[-stealth, line width=.6pt, line cap=round] (six.270) edge (seven);
\path[-stealth, line width=.6pt, line cap=round] (six.225) edge (g.55);
\path[-stealth, line width=.6pt, line cap=round] (seven.190) edge (f);
\path[-stealth, line width=.6pt, line cap=round] (seven.230) edge (g.20);
\end{tikzpicture}
\caption{The graph of \cite{felderwillwacher} contributing to $B_7$ with weight $\zeta(3)^2 /\pi^6$ up to rationals.}
\label{felderwillwachergraph}
\end{figure}
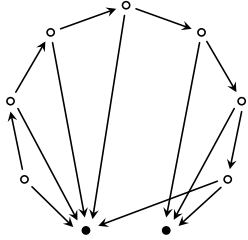

For the computations of cohomology and $\Ext^i$ groups in \S \ref{nchigherrank} we will not need the precise weights in the Kontsevich star product as it suffices to know the possible bidifferential operators that appear in its expression.

\subsection{Star products on $Z_k$}
\label{starproducts}

The cohomology of $Z_k$ can be described by the cohomology of the commutative diagram
\begin{align}
\label{opens}
\begin{tikzpicture}[baseline=-.35em,description/.style={fill=white,inner sep=2pt}]
\matrix (m) [matrix of math nodes, row sep={3em,between origins}, inner sep=2pt,
column sep={3em,between origins}, ampersand replacement=\&]
{\& Z_k \& \\ U \&\& V \\ \& U \cap V \& \\};
\path[-stealth,line width=.6pt,font=\scriptsize]
(m-3-2) edge (m-2-1)
(m-3-2) edge (m-2-3)
(m-2-1) edge (m-1-2)
(m-2-3) edge (m-1-2)
;
\end{tikzpicture}
\end{align}
where the arrows are inclusions of open sets.

Dual to (\ref{opens}), we have a commutative diagram of (commutative) algebras
\begin{align*}
\begin{tikzpicture}[description/.style={fill=white,inner sep=2pt}]
\matrix (m) [matrix of math nodes, row sep={3.25em,between origins}, inner sep=2pt,
column sep={3.25em,between origins}, ampersand replacement=\&]
{\& \mathcal O (Z_k) \& \\ \mathcal O (U) \&\& \mathcal O (V) \\ \& \mathcal O (U \cap V) \& \\};
\path[stealth-,line width=.6pt,font=\scriptsize]
(m-3-2) edge (m-2-1) 
(m-3-2) edge (m-2-3) 
(m-2-1) edge (m-1-2) 
(m-2-3) edge (m-1-2) 
;
\end{tikzpicture}
\end{align*}
where the arrows are the restriction morphisms.

To define a star product on $Z_k$, one has to find star products on $\mathcal O \llrr{\hbar} (W)$, where $W \in \{ Z_k, U, V, U \cap V \}$ making
\begin{align}
\label{diamond}
\begin{tikzpicture}[baseline=-.35em,description/.style={fill=white,inner sep=2pt}]
\matrix (m) [matrix of math nodes, row sep={3.5em,between origins}, inner sep=2pt,
column sep={3.5em,between origins}, ampersand replacement=\&]
{\& \mathcal O \llrr{\hbar} (Z_k) \& \\ \mathcal O \llrr{\hbar} (U) \&\& \mathcal O \llrr{\hbar} (V) \\ \& \mathcal O \llrr{\hbar} (U \cap V) \& \\};
\path[stealth-,line width=.6pt,font=\scriptsize]
(m-3-2) edge (m-2-1)
(m-3-2) edge (m-2-3)
(m-2-1) edge (m-1-2)
(m-2-3) edge (m-1-2)
;
\end{tikzpicture}
\end{align}
a commutative diagram of associative algebras.

We now describe quantizations of Poisson structures explicitly.

\section{Quantizing degenerate Poisson structures}
\label{quantizableimmersions}

To obtain explicit examples of deformation quantizations of Poisson structures on $Z_k$, we adapt the construction given in Kontsevich \cite{kontsevich2} of a {quantizable compactification} of a smooth affine Poisson variety. We show that the open immersion of a coordinate chart $U \subset Z_k$ is quantizable in case the Poisson structure is tangent to the divisor $D = Z_k \setminus U$,
see Def.~\ref{tangent}.

\begin{proposition}
\label{extends}
Let $\sigma $ be a Poisson structure on $Z_k$ such that is tangent to $D = Z_k \setminus U$. Then $U \subset Z_k$ is a quantizable open immersion, i.e.\ the Kontsevich star product on $U$ extends to a global star product on $Z_k$.
\end{proposition}

\begin{proof}
Let $D = Z_k \setminus U$ and write the algebras $\mathcal O (U \cap V), \mathcal O (U), \mathcal O (V), \mathcal O (Z_k)$ in the $\mathbb C$-vector space basis $\{ z^l u^i \}$ where $i \geq 0$ and
\begin{equation}
\label{boundszk}
\left\{
\begin{aligned}
-\infty < {} & l < \infty && \;\text{on $U \cap V$} \\
0 \leq {} & l < \infty && \;\text{on $U$} \\
-\infty < {} & l \leq ki && \;\text{on $V$} \\
0 \leq {} & l \leq ki && \;\text{on $U \cup V = Z_k$.}
\end{aligned}
\right.
\end{equation}
\paragraph{\bf The case $k = 1$.} As seen in Fig.~\ref{monomialspoissonzk}, a Poisson structure on $Z_1$ tangent to $D$ is generated by the monomial $1$ in $U$-coordinates, {\it i.e.}\ such a Poisson structure is of the form $\sigma_U = f$ for some global function $f \in \H^0 (Z_1, \mathcal O)$.

First assume that $f$ is constant. The Kontsevich star product on $U \simeq \mathbb C^2$ thus coincides with the Moyal product (\ref{moyal}). Denoting by $\star$ the restriction of this star product to $U \cap V$, the star product of two arbitrary monomials $z^{l_1} u^{i_1}, z^{l_2} u^{i_2} \in \mathcal O (U \cap V)$ is then given by
\begin{equation}
\label{starmonomialsz1}
z^{l_1} u^{i_1} \star z^{l_2} u^{i_2} = \sum_{n \geq 0} a_n \, z^{l_1 + l_2 - n} u^{i_1 + i_2 - n} \hbar^n
\end{equation}
for some constant coefficients $a_n \in \mathbb Q$ depending only on $l_1, l_2, i_1, i_2$. In particular,
\begin{align*}
a_0 &= 1 \\
a_1 &= l_1 i_2 - l_2 i_1 \\
a_2 &= \frac12 \big( l_1 (l_1 - 1) i_2 (i_2 - 1) - 2 l_1 l_2 i_1 i_2 + l_2 (l_2 - 1) i_1 (i_1 - 1) \big).
\end{align*}
If both $(l_1, i_1)$ and $(l_2, i_2)$ satisfy one and the same bound of (\ref{boundszk}), then so does $(\max \{ 0, l_1 + l_2 - n \}, i_1 + i_2 - n)$. Hence, $\star$ preserves the subalgebras $\mathcal O (Z_1), \mathcal O (U), $ and $\mathcal O (V)$ of $\mathcal O (U \cap V)$ and thus defines a global star product.

If $f$ is not constant, its power series expansion in $U$-coordinates is of the form $f = \sum_{i=0}^\infty \sum_{l=0}^{i} f_{il} z^l u^i$. (Note that $f$ being a global function means that the monomials appearing in the power series expansion of $f$ satisfy {\it all} of the bounds (\ref{boundszk}).) 

Now the Kontsevich star product also has contributions coming from graphs $\Gamma \in \mathfrak G_{n,2}$ with unfilled vertices representing a copy of the Poisson structure $\sigma_U = f$. However, an unfilled vertex lowers the powers of $z$ and $u$ each by $1$, but simultaneously multiplies by $\sigma_U = f$, so that the exponents still satisfy the bounds (\ref{boundszk}). \\

\paragraph{\bf The case $k \geq 2$.} Poisson structures tangent to $D$ are generated by the monomials $u, zu$ over global functions (see Prop.~\ref{tangentpoisson} and Fig.~\ref{monomialspoissonzk}). The proof is now the same as for $k = 1$ with $\sigma_U$ having vanishing constant term.
\end{proof}

\begin{remark}
The open immersion $U \subset Z_k$ satisfies some of the hypotheses of a quantizable compactification in the sense of \cite[Def.\ 4]{kontsevich2}, except for $Z_k$ not being compact and the divisor $D = Z_k \setminus U$ not being ample. However, the complement of $D$ is affine and even isomorphic to $\mathbb C^2$, which is the deciding factor for the construction. So, even though we do not have a compactification, we still obtain a quantization.

Given a quasi-projective variety $X$ and an affine subvariety $U \subset X$, one could of course consider a quantizable smooth compactification $U \subset X \subset \overline X$ and consider quantizations of Poisson structures which are tangent to $\overline X \setminus U$. However, if one is interested in quantizations of non-compact $X$, this strategy is more restrictive. For the case at hand, we have a smooth compactification $U \subset Z_k \subset F_k$ to the $k$th Hirzebruch surface, but the space of Poisson structures which are tangent to the divisor $F_k \setminus U$ is of dimension $3$ for $k = 1$ and of dimension $2$ for $k \geq 2$, whereas the space of quantizable Poisson structures for the open immersion $U \subset Z_k$ is in\-fi\-nite-dimen\-sional.
\end{remark}

\begin{example}
\label{moyalonV}
Let $\sigma$ be the Poisson structure $(1, -\xi)$ on $Z_1$ and consider the quantizable open immersion $U \subset Z_1$. Prop.~\ref{tangentpoisson} shows that $(1, -\xi)$ is tangent to $D = Z_1 \setminus U = \{ \xi = 0 \}$ and Prop.~\ref{extends} shows that there exists a global star product on $Z_1$, which can be computed in canonical coordinates by the Moyal product on $U$.

The corresponding star product on the $V$ chart is then given, up to second order, by:
\begin{align*}
f \star g = fg - \hbar &\Big( \xi \; \partial_\xi (f) \; \partial_v (g) - \xi \; \partial_v (f) \; \partial_\xi (g) \Big) \\
{} + \hbar^2 &\Big( \tfrac12 \xi^2 \; \partial_\xi^2 (f) \; \partial_v^2 (g) - \xi^2 \; \partial_\xi \partial_v (f) \; \partial_\xi \partial_v (g) + \tfrac12 \xi^2 \; \partial_v^2 (f) \; \partial_\xi^2 (g) \\
&\hspace{1em} + \xi \; \partial_\xi (f) \; \partial_v^2 (g) + \xi \; \partial_v (f) \; \partial_\xi \partial_v (g) \\
&\hspace{1em} + \xi \; \partial_\xi \partial_v (f) \; \partial_v (g) + \xi \; \partial_v^2 (f) \; \partial_\xi (g) \\
&\hspace{1em} - \partial_v (f) \; \partial_v (g) \\
&\hspace{1em} - v \; \partial_v (f) \; \partial_v^2 (g) - v \; \partial_v^2 (f) \; \partial_v (g)
\Big).
\end{align*}
To verify this, rewrite the Moyal product on $U$ using the identities
\begin{align*}
\partial_z &= -\xi^2 \partial_\xi + \xi v \partial_v \\
\partial_u &= \xi^{-1} \partial_v
\end{align*}
obtained by the change of coordinates and the commutation relations
\begin{equation}
\label{diffoprel}
[\partial_\xi, \xi] = \id, \qquad [\partial_v, v] = \id
\end{equation}
where $\xi, v$ in (\ref{diffoprel}) are thought of as differential operators of order $0$.
\end{example}

The following proposition shows that certain Poisson structures which in canonical coordinates do not satisfy the hypotheses of Prop.~\ref{extends} may still be quantized by the same method after performing a linear change of coordinates.

\begin{proposition}
Let $\sigma = (\sigma_U, \sigma_V)$ be a Poisson structure on $Z_k$ which on $U$ is of the form $\sigma_U = u^d P (z)$ for $d \geq 1$ and $P (z)$ a polynomial in $z$. Then there exists a linear change of coordinates $U \simeq U'$ such that $U' \subset Z_k$ is a quantizable open immersion.
\end{proposition}

\begin{proof}
First note that by Lem.~\ref{poissonstructures}, $P (z)$ is a polynomial of degree $n = (d-1)k + 2$. We can thus write
\[
\sigma_U = u^d \prod_{1 \leq i \leq n} (z - \lambda_i).
\]
Now, choose any $\lambda_i$ and consider the linear change of coordinates $U \simeq U'$ given by $(z, u) \mapsto (z', u)$ for $z' = z + \lambda_i$. In these coordinates we have that
\[
\sigma_U = z' u^d \prod_{i \neq j} (z' - \lambda_j).
\]
Then $\sigma$ is tangent to the divisor $\{ z' = 0 \}$ and $U' \subset Z_k$ is a quantizable open immersion.
\end{proof}

\begin{corollary}
For $k \neq 2$, any linear combination of the generators of Poisson structures on $Z_k$ can be quantized via an open immersion $\mathbb C^2 \simeq U' \subset Z_k$.
\end{corollary}

For the holomorphic symplectic structure on $Z_2$ the same method does not apply (see Kontsevich \cite[\S3.5]{kontsevich1} for more details).

\begin{remark}
\label{doesnotextendz2}
The Moyal product on $U \simeq \mathbb C^2$, does not extend naively to a star product on $Z_2$. In fact,  the Moyal product on $\mathcal O \llrr{\hbar} (U)$ does not preserve the subalgebra $\mathcal O \llrr{\hbar} (Z_k)$. Recall from (\ref{moyal}) that the bidifferential operators $B_n$ in the expression of the Moyal product on $\mathbb C^2_{z, u}$ are given by
\[
B_n = \frac{1}{n!} \sum_{0 \leq i \leq n} (-1)^i \binom{n}{i} \; \partial_z^{n-i} \partial_u^i (f) \; \partial_z^i \partial_u^{n-i} (g).
\]
Now consider the star product of the functions $z u, z^2 u \in \mathcal O (Z_2) \subset \mathcal O (U)$. Then
\begin{alignat*}{3}
z^2 u \star z u &= z^3 u^2 + {}& B_1 (z^2 u, z u) \, \hbar &{}+{}& B_2 (z^2 u, z u) \, \hbar^2 &{}+ \dotsb \\
&= z^3 u^2 + {}& 2 z^2 u \; \hbar &{}+{}& 2 z \; \hbar^2 &{}+ \dotsb
\end{alignat*}
But $2 z \notin \mathcal O (Z_2)$, so $z^2 u \star z u \notin \mathcal O \llrr{\hbar} (Z_k)$.
\end{remark}

\section{Geometry of commutative deformations}
\label{commutative}

In this section we summarize the known results about vector bundles and moduli for commutative deformations of $Z_k$ which will
be used to construct  the noncommutative counterparts  in \S\ref{vectorbundles}. 

As seen in \S \ref{poissongeometry}, holomorphic line bundles on $Z_k$ are classified by their first Chern class and we denote by $\mathcal O_{Z_k} (n)$ the line bundle with first Chern class $n$, omitting the subscript when it is clear from the context.
Recall that a rank $r$ bundle $E$ on $X$ is called {\it filtrable} if there exists an increasing filtration $0 = E_0 \subset E_1 \subset \dotsb \subset E_{r-1} \subset E_r = E$ of subbundles such that $E_i / E_{i-1} \in \Pic X$, where $1 \leq i \leq r$.

\begin{theorem}\cite[Lem.\ 3.1, Thm.\ 3.2]{gasparim}
\label{filtrable}
Holomorphic vector bundles on $Z_k$ are algebraic and filtrable.
\end{theorem}

\begin{definition}\cite{ballico}\label{type}
Let $E$ be a rank $r$ holomorphic vector bundle on $Z_k$. The restriction of $E$ to the zero section $\ell \simeq \mathbb P^1$ is a rank $r$ bundle on $\mathbb P^1$, which by Grothendieck's lemma splits as a direct sum of line bundles. Thus, $E \vert_\ell \simeq \mathcal O_{\mathbb P^1} (j_1) \oplus \dotsb \oplus \mathcal O_{\mathbb P^1} (j_r)$. We call $(j_1, \dotsc, j_r)$ the {\it splitting type} of $E$.
When $E$ is a rank $2$ bundle with first Chern class $0$, then the splitting type is $(j,-j)$ for some $j \geq 0$ and we say for short that $E$ has {\it splitting type} $j$.
\end{definition}

\begin{remark}\label{construct}
Filtrability of vector bundles on $Z_k$ implies that moduli of rank $2$ vector bundles are parametrized by classes in  $\Ext^1 (\mathcal O (j_2), \mathcal O (j_1))$ and algebraicity implies that such spaces of extensions are finite dimensional. For suitable numerical invariants or a suitable notion of stability, one may extract finite-dimensional moduli spaces from the naive quotient of the vector spaces $\Ext^1 (\mathcal O (j_2), \mathcal O (j_1))$ modulo bundle isomorphisms. 
\end{remark}

For applications to $\SU (2)$ instantons, one considers bundles with vanishing first Chern class. Thus, we study the quotient $\Ext^1 (\mathcal O (j), \mathcal O (-j)) / {\sim}$, where $\sim$ denotes bundle isomorphism. However, this quotient has an extremely complicated structure, in particular it is non-Hausdorff in the analytic topology. A stratification into Hausdorff components was presented in \cite[Thm.\ 4.15]{BGK2} by using two numerical invariants, whose sum makes up the local second Chern class (see Rem.~\ref{inv}).

To define the local Chern class, consider the (affine) surface $X_k$ obtained by contracting the zero section $\ell \subset Z_k$ to a point. Then $X_1 \simeq \mathbb C^2$ and for $k \geq 2$ one obtains the $\frac1k (1,1)$ surface singularity and $\tau \colon Z_k \rightarrow X_k$ is its (toric) resolution, which is given by inclusion of fans as shown in Fig.~\ref{fans}.

\begin{figure}
\begin{align*}
\begin{tikzpicture}[x=2em,y=2em]
\node at (.75,-.75) {{\it fan of $X_k$}};
\draw[line width=.6pt, line cap=round] (-1,3) -- (-1.16,3.48);
\draw[line width=.6pt, line cap=round] (1,0) -- (2.5,0);
\draw[-stealth,line width=.6pt, line cap=round] (0,0) -- (-.99,2.97);
\draw[-stealth,line width=.6pt, line cap=round] (0,0) -- (.97,0);
\draw[line width=.6pt, densely dotted] (-1,3) -- (0,3) -- (1,0);
\draw[line width=.6pt, fill=black] (-1,3) circle (.16ex);
\draw[line width=.6pt, fill=black] (0,0) circle (.16ex);
\draw[line width=.6pt, fill=black] (0,1) circle (.16ex);
\draw[line width=.6pt, fill=black] (0,2) circle (.16ex);
\draw[line width=.6pt, fill=black] (0,3) circle (.16ex);
\draw[line width=.6pt, fill=black] (1,0) circle (.16ex);
\draw[line width=.6pt, fill=black] (1,1) circle (.16ex);
\draw[line width=.6pt, fill=black] (1,2) circle (.16ex);
\draw[line width=.6pt, fill=black] (1,3) circle (.16ex);
\draw[line width=.6pt, fill=black] (2,0) circle (.16ex);
\draw[line width=.6pt, fill=black] (2,1) circle (.16ex);
\draw[line width=.6pt, fill=black] (2,2) circle (.16ex);
\begin{scope}[shift={(5,0)}]
\node at (.75,-.75) {{\it fan of $Z_k$}};
\draw[line width=.6pt, line cap=round] (-1,3) -- (-1.16,3.48);
\draw[line width=.6pt, line cap=round] (1,0) -- (2.5,0);
\draw[-stealth,line width=.6pt, line cap=round] (0,0) -- (-.99,2.97);
\draw[-stealth,line width=.6pt, line cap=round] (0,0) -- (.97,0);
\draw[-stealth,line width=.6pt, line cap=round] (0,0) -- (0,.97);
\draw[line width=.6pt, line cap=round] (0,1) -- (0,3.5);
\draw[line width=.6pt, fill=black] (-1,3) circle (.16ex);
\draw[line width=.6pt, fill=black] (0,0) circle (.16ex);
\draw[line width=.6pt, fill=black] (0,1) circle (.16ex);
\draw[line width=.6pt, fill=black] (0,2) circle (.16ex);
\draw[line width=.6pt, fill=black] (0,3) circle (.16ex);
\draw[line width=.6pt, fill=black] (1,0) circle (.16ex);
\draw[line width=.6pt, fill=black] (1,1) circle (.16ex);
\draw[line width=.6pt, fill=black] (1,2) circle (.16ex);
\draw[line width=.6pt, fill=black] (1,3) circle (.16ex);
\draw[line width=.6pt, fill=black] (2,0) circle (.16ex);
\draw[line width=.6pt, fill=black] (2,1) circle (.16ex);
\draw[line width=.6pt, fill=black] (2,2) circle (.16ex);
\end{scope}
\end{tikzpicture}
\end{align*}
\caption{$Z_k$ as toric resolution of $X_k$}
\label{fans}
\end{figure}
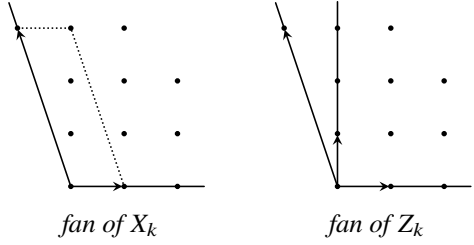

\begin{remark}
Observe that $X_k \simeq \mathbb C^2 / \Gamma$, where $\Gamma \subset \GL (2, \mathbb C)$ is a cyclic group of order $k$ with a generator acting on $\mathbb C^2$ via multiplication by $\gamma = \left( \begin{smallmatrix} \omega & 0 \\ 0 & \omega \end{smallmatrix} \right)$ for $\omega$ a primitive $k$th root of unity. Thus $X_2$ is the $A_1$ surface singularity, but for $k \geq 3$ we have that $\det \gamma = \omega^2 \neq 1$. In particular, $\Gamma \not\subset \SU (2)$ so that $Z_{\geq 3}$ is not an ALE space, see \cite[Thm.\ 1.2]{Kr}.
\end{remark}

\begin{definition}
\label{charge}
Let $E$ be a holomorphic rank $2$ bundle on $Z_k$ and let $\tau \colon Z_k \to X_k$ be the contraction map. The {\it local second Chern class} of $E$ is the local holomorphic Euler characteristic of $E$ around $\ell$, that is,
\begin{equation}
\label{numinv}
\chi (\ell, E) = h^0 (X_k, (\tau_*E)^{\vee\vee}/\tau_* E) + h^0 (X_k, \mathrm R^1\tau_*E).
\end{equation}
\end{definition}

\begin{remark}
\label{inv}
The terms on the right-hand side of (\ref{numinv}) define two independent holomorphic invariants of the vector bundle $E$. In \cite{BGK2} these invariants were called the {\it width} of $E$, since $h^0 (X_k, (\tau_* E)^{\vee\vee}/\tau_* E)$ measures the default of the direct image from being locally free, and the {\it height} of $E$, since $h^0 (X_k, \mathrm R^1 \tau_* E)$ measures how far $E$ is from being a split extension. 
These are two independent numerical invariants, and the pair stratifies the moduli spaces $\mathfrak M_j (Z_k)$ into Hausdorff components \cite[Thm.\ 4.15]{BGK2}.
\end{remark}

To describe actual moduli spaces,  \cite[Def.~5.2]{GKM} give an ad-hoc definition of stability, calling a rank $2$ vector bundle on $Z_k$ {\it (framed) stable} when it is holomorphically trivial (and framed) on $Z_k \setminus \ell$.

\begin{notation}
\label{modulic}
Denote by $\mathfrak M_j (Z_k)$ the subspace of $\Ext^1(\mathcal O_{Z_k}(j), \mathcal O_{Z_k}(-j))/{\sim}$ consisting of those classes corresponding to stable vector bundles, where $\sim$ denotes bundle isomorphism. 
\end{notation}

Moduli spaces of rank~$2$ bundles on $Z_k$ were studied in 
\cite[Thm.\ 3.5]{gasparim2} for the case $k = 1$ and in \cite[Thm.\ 4.11]{BGK2} for the cases $k \geq 1$. The moduli spaces of stable bundles with splitting type $j$ on $Z_k$,  turn out to be smooth quasi-projective varieties of dimension ${2j-k-2}$ \cite[Thm.\ 4.11]{BGK2}. In fact, we have:

\begin{theorem}\label{exist}
\cite[Thm.\ 4.11]{BGK2}
The moduli space of rank $2$ holomorphic bundles on $Z_k$ with vanishing first Chern class and 
splitting type $j$ contains an open dense subset isomorphic to $\mathbb P^{2j-k-2}$ minus a closed subvariety of codimension at least $k + 1$.
\end{theorem}

\section{Geometry of noncommutative deformations}
\label{vectorbundles}

We now study vector bundles over noncommutative deformations $\mathcal Z_k (\sigma) = (Z_k, \mathcal A^\sigma)$ of $Z_k$ (see Def.~\ref{dq}).

\begin{definition}\label{vb}
We call a locally free sheaf of $\mathcal A^\sigma$-modules of rank $r$ over $\mathcal Z_k (\sigma)$ a {\it vector bundle of rank $r$ over $\mathcal Z_k (\sigma)$}. A vector bundle of rank $1$ over $\mathcal Z_k (\sigma)$ is called a {\it line bundle}.
\end{definition}

We recall some properties of deformation quantizations and refer to Kashiwara--Schapira \cite{kashiwaraschapira} for more details.

\begin{proposition} \cite{kashiwaraschapira}
Let $(X, \sigma)$ be a holomorphic Poisson manifold and let $\mathcal A = \mathcal A^\sigma$ be a deformation quantization of $\mathcal O_X$.
\begin{enumerate}
\item Let $\mathcal E$ be a coherent sheaf of $\mathcal A$-modules without $\hbar$-torsion. If $\mathcal E / \hbar \mathcal E$ is locally free of rank~$r$ as a sheaf of $\mathcal O_X$-modules, then $\mathcal E$ is locally free of rank~$r$ as a sheaf of $\mathcal A$-modules.
\item If $U \subset X$ is affine, {\it i.e.}\ $\H^k (U, \mathcal F\vert_U) = 0$ for any $k > 0$ and for any coherent sheaf $\mathcal F$ of $\mathcal O_X$-modules, then $\H^k (U, \mathcal G) = 0$ for any coherent sheaf $\mathcal G$ of $\mathcal A$-modules.
\end{enumerate}
\end{proposition}

Here (\textit{ii}\hair) says that a Leray cover for $(X, \mathcal O_X)$ is also a Leray cover for $(X, \mathcal A)$. In particular, we may use the canonical coordinate charts on $Z_k$ to calculate cohomology or extension groups of coherent (or locally free) sheaves of $\mathcal A$-modules. After showing that rank $2$ bundles are extensions of line bundles (Thm.~\ref{ncfiltrable}), we will calculate these extension groups and use them to obtain moduli of vector bundles over noncommutative deformations in \S \ref{ncmoduli}.

In terms of canonical coordinate charts for $Z_k$ Def.~\ref{vb} implies that a rank $r$ vector bundle over a noncommutative $\mathcal Z_k (\sigma)$ is given by two free rank $r$ modules over $U$ and over $V$ with a global structure defined by a {\it transition matrix}, {\it i.e.}\ a $\star$-invertible $r \times r$ matrix with entries in $\mathcal A^\sigma (U \cap V)$, determining the vector bundle uniquely up to isomorphism, where an isomorphism of vector bundles on $\mathcal Z_k (\sigma)$ is an isomorphism of $\mathcal A^\sigma$-modules, which can be phrased using the coordinate charts of $Z_k$ as follows.
 
\begin{definition}
\label{nciso}
Let $E$ and $E'$ be vector bundles over $\mathcal Z_k (\sigma)$ defined by transition matrices $T$ and $T'$ respectively. An {\it isomorphism} between $E$ and $E'$ is given by a pair of matrices $A_U$ and $A_V$ with entries in $\mathcal A^\sigma (U)$ and $\mathcal A^\sigma (V)$, respectively, which are invertible with respect to $\star$ and such that
\[
T'= A_V \star T \star A_U.
\]
\end{definition}

\begin{definition}
Let $X$ be a complex manifold (resp.\ smooth algebraic variety) and let $\mathcal A$ be a noncommutative associative deformation of $\mathcal O_X$ over $\mathbb C \llrr{\hbar}$. The augmentation $\mathcal A \to \mathbb C \otimes_{\mathbb C \llrr{\hbar}} \mathcal A \simeq \mathcal O_X$ is given by $\mathcal A \to \mathcal A / \hbar \mathcal A$. Augmentation induces a map on quasi-coherent sheaves of $\mathcal A$-modules and the image of a sheaf $\mathcal F$ of $\mathcal A$-modules is called the {\it classical limit} of $\mathcal F$. The image of any cocycle or cohomology class $\alpha \in \H^i (X, \mathcal F (U))$, for $U \subset X$ open, is called the {\it classical limit} of $\alpha$.
\end{definition}

\subsection{Line bundles}

\begin{lemma}
\label{acyclic}
Let $\mathcal A$ be a deformation quantization of $\mathcal O$. Then an $\mathcal A$-module $\mathcal S$ is acyclic if and only $S = \mathcal S / \hbar \mathcal S$ is acyclic.
\end{lemma}

\begin{proof}
Consider the short exact sequence 
\[
0 \longrightarrow \mathcal S \stackrel{\hbar}\longrightarrow \mathcal S \longrightarrow S \longrightarrow 0.
\]
It gives, for $j > 0$ surjections 
\[
\H^j (X, \mathcal S) \stackrel{\hbar}\longrightarrow \H^j (X, \mathcal S) \longrightarrow 0.
\]
This immediately implies that $\H^j (X, \mathcal S) = 0$ for $j > 0$. The converse is immediate.
\end{proof}

\begin{definition}
Let $\mathcal Z_k (\sigma)$ be a noncommutative deformation of $Z_k$. Denote by $\mathcal A (j)$ the line bundle over $\mathcal Z_k (\sigma)$ with transition function $z^{-j}$.
\end{definition}

\begin{proposition}
\label{nclinebundles}
Any line bundle on $\mathcal Z_k (\sigma)$ is isomorphic to $\mathcal A (j)$ for some $j \in \mathbb Z$, {\it i.e.}\ $\Pic (\mathcal Z_k (\sigma)) \simeq \mathbb Z$.
\end{proposition}

\begin{proof}
Let $f = f_0 + \sum_{n=1}^\infty \widetilde f_n \,\hbar^n \in \mathcal A^* (U \cap V)$ be the transition function for $\mathcal L$. Then there exist functions $a_0 \in \mathcal O^* (U)$ and $\alpha_0 \in \mathcal O^* (V)$ such that $\alpha_0 f_0 a_0 = z^{-j}$ and viewing $a_0$ resp.\ $\alpha_0$ as elements in $\mathcal A^* (U)$ resp.\ $\mathcal A^* (V)$ one has $\alpha_0 \star f \star a_0 = z^{-j} + \sum_{n=1}^\infty f_n \hbar^n$ for some $f_n \in \mathcal O (U \cap V)$. We may thus assume that the transition function of $\mathcal L$ is 
$z^{-j} + \sum_{n=1}^\infty f_n \hbar^n.$

To give an isomorphism $\mathcal L \simeq \mathcal A (j)$ it suffices to define functions $a_n \in \mathcal O (U)$ and $\alpha_n \in \mathcal O (V)$ satisfying
\begin{align}
\label{linebundleisomorphism}
\big( 1 + \textstyle\sum_{n=1}^\infty \alpha_n \hbar^n \big) \star \big( z^{-j} + \sum_{n=1}^\infty f_n \hbar^n \big) \star \big( 1 + \sum_{n=1}^\infty a_n \hbar^n \big) = z^{-j}.
\end{align}

Collecting terms by powers of $\hbar$, (\ref{linebundleisomorphism}) is equivalent to the system of equations
\begin{align}
S_n + z^{-j} a_n + z^{-j} \alpha_n = 0 \tag*{$n = 1, 2, \dotsc$}
\end{align}
where $S_n$ is a finite sum involving $f_i$, $B_i$ for $i \leq n$, but only $a_i, \alpha_i$ for $i < n$. The first terms are
\begin{align*}
S_1 &= f_1 \\
S_2 &= f_2 + \alpha_1 f_1 + a_1 f_1 + B_1 \big( \alpha_1, z^{-j} \big) + B_1 \big( z^{-j}, a_1 \big) + \alpha_1 z^{-j} a_1 \\
S_3 &= f_3
+ B_2 \big( \alpha_1, z^{-j} \big)
+ B_2 \big( z^{-j}, a_1 \big)
+ B_1 \big( \alpha_2, z^{-j} \big) \\ &\qquad
+ B_1 \big( z^{-j}, a_2 \big)
+ B_1 \big( \alpha_1, f_1 \big)
+ B_1 \big( \alpha_1, z^{-j} a_1 \big)
+ B_1 \big( z^{-j}, a_1 \big) \\ &\qquad
+ \alpha_2 f_1 + \alpha_2 z^{-j} a_1 + \alpha_1 f_2 + \alpha_1 f_1 a_1 + \alpha_1 z^{-j} a_2 + f_2 a_1 + f_1 a_2
\end{align*}
Since $\H^1 (Z_k, \mathcal O) = 0$ we can solve these equations recursively, for example by defining $a_n$ to cancel all terms of $z^j S_n$ having positive powers of $z$ and setting $\alpha_n = z^j S_n - a_n$.
\end{proof}

\subsection{Vector bundles}
\label{nchigherrank}

Generalizing Thm.~\ref{filtrable} to the noncommutative setting, we prove filtrability and formal algebraicity for bundles over deformation quantizations of $Z_k$. Let $\mathcal Z_k (\sigma)$ be a noncommutative deformation of $Z_k$.

\begin{theorem}
\label{ncfiltrable}
 Vector bundles over $\mathcal Z_k (\sigma)$ are filtrable.
\end{theorem}

\begin{proof}[First proof.]
Let $\mathcal E$ be the rank $2$ bundle given by transition matrix $T = T_0 + \sum_{n=1}^\infty T_n \hbar^n$. Let $(A_U, A_V)$ be an isomorphism of the classical limit $\mathcal E / \hbar \mathcal E$ (with transition matrix $T_0$) with a filtered bundle, {\it i.e.}\
\[
A_V T_0 A_U = 
\begin{pmatrix}
z^{j_1} & b_0 \\
 0 & z^{j_2}
\end{pmatrix}
\]
with $j_1 \geq j_2$.
Then $(A_U, A_V)$ gives an isomorphism with the bundle given by
\begin{align}\label{transitionrank2}
A_V \star T \star A_U =
\begin{pmatrix}
z^{j_1} + O (\hbar) & b_0 + O (\hbar) \\
O (\hbar) & z^{j_2} + O (\hbar)
\end{pmatrix} = T'\text{.}
\end{align}
Now choose $\alpha_n, \delta_n$ and $a_n, d_n$ as in the proof of Prop.~\ref{nclinebundles} such that the $(1,1)$ and $(2,2)$ entries of (\ref{transitionrank2}) are taken to the transition functions of $\mathcal A (-j_1)$ and $\mathcal A (-j_2)$, respectively. Then for
\begin{align*}
A'_V &=
\begin{pmatrix}
1 + \sum_{n=1}^\infty \alpha_n \hbar^n & 0 \\
0 & 1 + \sum_{n=1}^\infty \delta_n \hbar^n
\end{pmatrix}\\
A'_U &=
\begin{pmatrix}
1 + \sum_{n=1}^\infty a_n \hbar^n & 0 \\
0 & 1 + \sum_{n=1}^\infty d_n \hbar^n
\end{pmatrix}
\end{align*}
we have that
\begin{equation*}
A'_V \star T' \star A'_U =
\begin{pmatrix}
z^{j_1} & b_0 + O (\hbar) \\
O \big( \hbar^N \big) & z^{j_2}
\end{pmatrix}
=
\begin{pmatrix}
z^{j_1} & b_0 + O (\hbar) \\
\sum_{n=N}^\infty c_n \hbar^n & z^{j_2}
\end{pmatrix}
=
T''
\end{equation*}
for some integer $N \geq 1$.
Now choose $c'_N$ and $\gamma'_N$ such that
\[
z^{j_1} \gamma'_N + c_N + z^{j_2} c'_N = 0
\]
which is possible since $\H^1 (Z_k, \mathcal O (j_1 {-} j_2)) = 0$, as $j_1 \geq j_2$. Then
\[
\begin{pmatrix}
1 & 0 \\
\gamma'_N \hbar^N & 1
\end{pmatrix}
\star
T''
\star
\begin{pmatrix}
1 & 0 \\
c'_N \hbar^N & 1
\end{pmatrix}
=
\begin{pmatrix}
z^{j_1} + O (\hbar) & b_0 + O (\hbar) \\
O \big( \hbar^{N+1} \big) & z^{j_2} + O (\hbar)
\end{pmatrix}
=
T'''.
\]
As for the isomorphism between the bundles defined by $T'$ and $T''$, find matrices taking the $(1,1)$ and $(2,2)$ entries of $T'''$ to $z^{j_1}$ and $z^{j_2}$ as above. We thus get that an isomorphic bundle may be given by a transition matrix of the form
\[
\begin{pmatrix}
z^{j_1} & b_0 + O (\hbar) \\
O \big( \hbar^{N+1} \big) & z^{j_2}
\end{pmatrix}\text{.}
\]

Applying the principle of (strong) induction, we conclude that $\mathcal E$ is isomorphic to the bundle given by a transition matrix
\[
\begin{pmatrix}
z^{j_1} & b_0 + \sum_{n=1}^\infty b_n \hbar^n \\
0 & z^{j_2}
\end{pmatrix}\text{.}
\]
In particular, any rank $2$ bundle is an extension of line bundles. The calculation for rank $r$ is similar.
\end{proof}

\begin{proof}[Second proof.]
This is a generalization of Ballico--Gasparim--K\"oppe \cite[Thm.~3.2]{ballicogasparimkoppe1} to the noncommutative case. Let  
$\mathcal E$ be a sheaf of $\mathcal A$-modules. 
Lem.~\ref{acyclic} gives that the classical limit
$\mathcal E_0 = \mathcal E / \hbar \mathcal E$ is acyclic as a sheaf of $\mathcal A$-modules (and equivalently as a sheaf of $\mathcal O$-modules) if and only if $\mathcal E$ is acyclic as a sheaf of $\mathcal A$-modules. 
 
Filtrability for a bundle $E$ over $Z_k$ is obtained due to the vanishing of cohomology groups $\H^i (Z_k, E \otimes {\Sym^n} N^*)$ for $i=1,2$, where $N^*$ is the conormal bundle of $\ell \subset Z_k$ and $n>0$ are integers, the proof proceeds by induction on $n$. In the noncommutative case, let $\mathcal S$ denote the kernel of the projection $\mathcal A^{(n)} \to \mathcal A^{(n-1)}$. By construction we have that $\mathcal S / \hbar \mathcal S = {\Sym^n} N^*$ and the required vanishing of cohomologies is guaranteed by Lem.~\ref{acyclic}. \end{proof}

As for $Z_k$, rank $2$ vector bundles on $\mathcal Z_k (\sigma)$ are thus extensions of line bundles.
Hence, moduli spaces of rank $2$ bundles may be built out of  quotients of $\Ext^1$'s.

\begin{definition}\label{canonical}
\cite[Thm.\ 3.3]{gasparim} showed that a rank $2$ bundle $E$ on $Z_k$ 
with first Chern class $c_1 (E) = 0$ can be given by a {\it canonical transition matrix}
\begin{align*}
T_0 &=
\begin{pmatrix}
z^j & p \\
 0  & z^{-j}
\end{pmatrix}
\qquad \mbox{with} \qquad
p = \sum_{i=0}^{\left\lfloor\! \frac{2j - 2}{k} \!\right\rfloor} \sum_{l = ki-j+1}^{j-1} p_{il} z^l u^i  \in \Ext^1 (\mathcal O (j), \mathcal O (-j)).
\intertext{Accordingly, for a noncommutative deformation $\mathcal Z_k (\sigma)$ we define the corresponding notion of {\it canonical  transition matrix} as:}
T &=
\begin{pmatrix}
z^j & {\bf p} \\
 0  & z^{-j}
\end{pmatrix}
\qquad \mbox{with} \qquad
{\bf p} = \sum_{n=0}^\infty p_n \hbar^n 
 \in \Ext^1 (\mathcal A (j), \mathcal A (-j)).
\end{align*}
\end{definition}

We will now see that each $p_n$ can be given the canonical form of the classical case. 

\begin{lemma}
\label{subspace}
Let $\mathcal A$ be a deformation quantization of $\mathcal O_{Z_k}$. There is an injective map of $\mathbb C$-vector spaces
\begin{align*}
\begin{tikzpicture}[baseline=-2.6pt,description/.style={fill=white,inner sep=2pt}]
\matrix (m) [matrix of math nodes, row sep=1em, text height=1.5ex, column sep=1.5em, text depth=0.25ex, ampersand replacement=\&, column 2/.style={anchor=base west}]
{
\Ext^1_{\mathcal A} (\mathcal A (j), \mathcal A (-j)) \& \displaystyle\prod_{n=0}^\infty \Ext^1_{\mathcal O} (\mathcal O (j), \mathcal O (-j)) \hbar^n \simeq \Ext^1_{\mathcal O} (\mathcal O (j), \mathcal O (-j)) \llrr{\hbar} \\
{\bf p} = p_0 + \displaystyle\sum_{n=1}^\infty p_n \hbar^n \& (p_0, p_1 \hbar, p_2 \hbar^2, \dotsc) \\
}
;
\path[|-stealth,line width=.5pt,font=\scriptsize]
(m-2-1) edge (m-2-2)
;
\path[-stealth,line width=.5pt,font=\scriptsize]
(m-1-1) edge (m-1-2)
;
\end{tikzpicture}
\end{align*}
where $p_i \in \Ext^1 (\mathcal O (j), \mathcal O (-j))$.
\end{lemma}

\begin{proof}[Short proof]
$\Ext^1_{\mathcal A} (\mathcal A (j), \mathcal A (-j)) $ is the quotient of $\Ext^1_{\mathcal O} (\mathcal O (j), \mathcal O (-j)) \llrr{\hbar}$
by the relations $q_n \simeq q_n + \sum p_i p_{n-i}$. 
\end{proof}

\begin{proof}
Two extension classes ${\bf p} = p_0 + \sum_{n=1}^\infty p_n \hbar^n$ and ${\bf p}' = p'_0 + \sum_{n=1}^\infty  p'_n \hbar^n$ are equivalent if there exist two functions
\begin{equation}
\label{bbeta}
b     = \eqmakebox[bbeta0]{$b_0$}     + \sum_{n=1}^\infty \eqmakebox[bbetan]{$b_n$} \hbar^n \in \mathcal A (U) \quad \text{and} \quad
\beta = \eqmakebox[bbeta0]{$\beta_0$} + \sum_{n=1}^\infty \eqmakebox[bbetan]{$\beta_n$} \hbar^n \in \mathcal A (V)
\end{equation}
such that
\begin{flalign}
&&\begin{pmatrix}
1 & \beta \\ 0 & 1
\end{pmatrix}
\star
\begin{pmatrix}
z^j & {\bf p} \\
0 & z^{-j}
\end{pmatrix}
&=
\begin{pmatrix}
z^j & {\bf p}' \\
0 & z^{-j}
\end{pmatrix}
\star
\begin{pmatrix}
1 & b \\ 0 & 1
\end{pmatrix}\text{.} && \notag
\intertext{that is}
&& {\bf p} + \beta \star z^{-j} &= {\bf p}' + z^j \star b \label{extensions} &&
\intertext{or equivalently}
&& p_n + \beta_n z^{-j} + \sum_{i = 1}^n B_i (\beta_{n-i}, z^{-j}) &= p'_n + z^j b_n + \sum_{i = 1}^n B_i (z^j, b_{n-i}) && \mathllap{\text{for }} n \in \mathbb N. \notag
\end{flalign}
For each $n$ we may choose $b_n$ and $\beta_n$ reducing $p_n$ to the canonical form
\begin{equation}
\label{formpn}
p_n = \sum_{i=0}^{\left\lfloor\! \frac{2j-2}k \!\right\rfloor} \sum_{l=ik-j+1}^{j-1} p^n_{il} z^l u^i
\end{equation}
as in (\ref{canonical}).
\end{proof}

\begin{lemma}\label{z1}
Let $\mathcal Z_1 (\sigma)$ be the deformation quantization of $Z_1$ with Poisson structure $\sigma = (1, -\xi)$. Then any ${\bf p} \in \Ext^1_{\mathcal A} (\mathcal A (j), \mathcal A (-j))$ is represented in the canonical form
${\bf p} = \sum_{n=0}^\infty p_n \hbar^n
$
where $p_n$ is of the form
$
p_n = \sum_{i=0}^{2j-2} \sum_{l=i-j+1}^{j-1} p^n_{il} z^l u^i.
$
\end{lemma}

\begin{proof}
On the $U$ chart, the holomorphic Poisson structure $\sigma$ is induced by the bivector $\frac{\partial}{\partial z} \wedge \frac{\partial}{\partial u}$, and the corresponding star product is given by the Moyal product (\ref{moyal}).

As in the proof of Lem.~\ref{subspace}, elements ${\bf p} = p_0 + \sum_{n=1}^\infty p_n \hbar^n$ and ${\bf p}' = p'_0 + \sum_{n=1}^\infty p'_n \hbar^n$
in $\Ext^1 (\mathcal A (j), \mathcal A (-j))$ are equivalent if there exist two functions $b \in \mathcal A (U)$ and $\beta \in \mathcal A (V)$ of the form (\ref{bbeta}) such that
\[
{\bf p}'= {\bf p} + \beta \star z^{-j} - z^j \star b.
\]
We carry out calculations on the $U$ chart, hence using the Moyal product.
For $n = 1,2$ this gives:
\begin{align*}
 p'_1 = p_1 &{}+ \beta_1 z^{-j} - \big( \tfrac{\partial}{\partial u} \beta_0 \big) \big( \tfrac{\partial}{\partial z} z^{-j} \big) - z^j b_1 + \big( \tfrac{\partial}{\partial z} z^j \big) \big( \tfrac{\partial}{\partial u} b_0 \big) \\
p_2' = p_2 &{}+ \beta_2 z^{-j} - \big( \tfrac{\partial}{\partial u} \beta_1 \big) \big( \tfrac{\partial}{\partial z} z^{-j} \big) - \big( \tfrac{\partial^2}{\partial u^2} \beta_0 \big) \big( \tfrac{\partial^2}{\partial z^2} z^{-j} \big) \notag\\
&{} - z^j b_2 + \big( \tfrac{\partial}{\partial z} z^j \big) \big( \tfrac{\partial}{\partial u} b_1 \big) + \big( \tfrac{\partial^2}{\partial z^2} z^j \big) \big( \tfrac{\partial^2}{\partial u^2} b_0 \big)
\end{align*}
etc. But we then find out that ultimately the calculations repeat at each step the same type of calculation done on neighbourhood zero, thus at each step the shape of the polynomial is the same as what we have in the classical limit. 
\end{proof}

\begin{theorem}
\label{ncextformal}
Let $\sigma$ be a Poisson structure on $Z_k$ tangent to the divisor $Z_k \setminus U$ and let $\mathcal A$ be the quantization of the open immersion $U \subset Z_k$. Then
\[
\Ext^1_{\mathcal A} (\mathcal A (j), \mathcal A (-j)) \simeq \Ext^1_{\mathcal O} (\mathcal O (j), \mathcal O (-j)) \llrr{\hbar}.
\]
\end{theorem}

\begin{proof}
Recall from the proof of Lem.~\ref{subspace} that an extension class ${\bf p} \in \Ext^1_{\mathcal A} (\mathcal A (j), \mathcal A (-j))$ may be reduced to the form ${\bf p} = p_0 + \sum_{n=1}^\infty p_n \hbar^n$, where $p_i \in \Ext^1_{\mathcal O} (\mathcal O (j), \mathcal O (-j))$. It remains to show that this form cannot be reduced further.

For $Z_1$, this is the content of Lem.~\ref{z1}. For $Z_{\geq 2}$, Prop.~\ref{tangentpoisson} shows that for $\sigma$ tangent to $D$, $\sigma_U = u \, f_U + zu \, g_U$ for some global functions $f_U, g_U$. Since $\sigma_U$ is a multiple of $u$, it follows that the bilinear operators $B_n$ in the expression of the star product never lower the exponents of $u$, {\it cf.}\ the proof of Prop.~\ref{extends}. Thus the star products in (\ref{extensions}) never contain terms with exponents of $u$ low enough to reduce the general form (\ref{formpn}) of $p_n$ any further.
\end{proof}

\begin{definition}
\label{formally}
We say that ${\bf p} = \sum p_n \hbar^n \in \mathcal O \llrr{\hbar}$ is formally algebraic if $p_n$ is a polynomial for every $n$.

We say that a vector bundle over $\mathcal Z_k (\sigma)$ is {\it formally algebraic} if it is isomorphic to a vector bundle given by formally algebraic transition functions. In addition, if there exists $N$ such that $p_n=0$ for all $n>N$, we then say that ${\bf p}$ is {\it algebraic}.
\end{definition}

\begin{corollary}\label{formalg}
Vector bundles on noncommutative deformations of $Z_k$ are formally algebraic.
\end{corollary}

\begin{proof}
Lemma~\ref{subspace} shows that rank $2$ vector bundles over $\mathcal Z_k (\sigma)$ are formally algebraic, and in light of Thm.~\ref{ncfiltrable}, we obtain the result for all ranks.
\end{proof}

\begin{remark}
A quantization of $Z_k$ will also give a quantization of the singular affine surface $X_k = \Spec (\H^0 (Z_k, \mathcal O))$, obtained from $Z_k$ by contracting the zero section $\ell$ to a point. ($X_k$ contains an isolated $\frac1k (1,1)$ singularity, see \S\ref{instantons} and Fig.~\ref{fans}.) Indeed, a quantization of $Z_k$ gives a star product on $\mathcal O_{Z_k} (U)$ for all open sets $U$, and thus in particular a star product on the algebra of global functions $\H^0 (Z_k, \mathcal O) \simeq \H^0 (X_k, \mathcal O)$.

Note that Poisson structures on singular affine {\it toric} varieties (of which $X_k$ is a particular case) are known to always admit quantizations (see Filip \cite{filip}), but in general there are also counterexamples to the quantization of singular Poisson algebras (see Mathieu \cite{mathieu} and also Schedler \cite{schedler}).
\end{remark}

Since $\Ext^1_{Z_k} (\mathcal O (j), \mathcal O (-j))$ is finite dimensional, we have that if $\mathcal Z_k (\sigma)$ is a noncommutative deformation having finite order in $\hbar$, then the space of extensions $\Ext^1_{\mathcal A} (\mathcal A (j), \mathcal A (-j))$ is also finite dimensional. For moduli of vector bundles on $\mathcal Z_k (\sigma)$ up to all orders of $\hbar$ see Rem.~\ref{higherhbar}.

\section{Moduli of bundles on noncommutative deformations}
\label{ncmoduli}

We now calculate moduli spaces for vector bundles on deformation quantizations of $(Z_k, \sigma)$ for $\sigma$ a holomorphic Poisson structure.

\begin{remark}
Analogous to the situation of Rem.~\ref{construct}, one has filtrability of vector bundles also in the noncommutative case (Lem.~\ref{ncfiltrable}) so that moduli of rank $2$ vector bundles can be described as quotients of $\Ext^1$ of line bundles. Furthermore, we have formal algebraicity (Cor.~\ref{formalg}) as the extension groups are $\mathbb N$-graded by powers of $\hbar$ with finite-dimensional graded components, or indeed finite-dimensional if we take $\hbar$ only up to a fixed finite power. We then obtain finite-dimensional quotients taking $\Ext^1$ modulo bundle isomorphisms, where here isomorphisms are defined in Def.~\ref{nciso} using the (noncommutative) star product. 
\end{remark}

We thus may proceed as in the classical (commutative) setting and extract moduli spaces from extension groups of line bundles, by considering extension classes up to bundle isomorphism.
 
\begin{notation} 
\label{modulinc}
We denote by $\mathfrak M_j (\mathcal Z_k (\sigma))$ the subspace of the quotient \mbox{$\Ext^1_{\mathcal A} (\mathcal A (j), \mathcal A (-j))/{\sim}$} consisting of those classes of formally algebraic vector bundles, whose classical limit is a stable vector bundle of charge $j$. Here $\sim$ denotes bundle isomorphism as in Def.~\ref{nciso}. We denote by $\mathfrak M_j^{(n)} (\mathcal Z_k (\sigma))$ the moduli of bundles obtained by imposing the cut-off $\hbar^{n+1} = 0$, that is, the superscript $^{(n)}$ means quantized to level $n$. Accordingly $\mathfrak M_j^{(1)} (\mathcal Z_k (\sigma))$ stands for first-order quantization and $\mathfrak M_j^{(0)} (\mathcal Z_k (\sigma)) = \mathfrak M_j (Z_k)$ recovers the classical moduli space of Def.~\ref{modulic} obtained when $\hbar = 0$.
\end{notation}

\begin{definition}\label{split}
The {\it splitting type} of a vector bundle $E$ on $\mathcal Z_k (\sigma)$ is defined to be the splitting type of its classical limit as in Def.~\ref{type}. Hence, when the classical limit is an $\mathrm{SL}(2,\mathbb C)$ bundle, the splitting type of $E$ is the smallest integer $j$ such that $E$ can be written as an extension of $\mathcal A (j)$ by $\mathcal A (-j)$. 
\end{definition}

We will look at rank $2$ bundles of a fixed splitting type $j$ on the first formal neighbourhood $\ell^{(1)}$ of $\ell \subset Z_k$. In order for the relevant space of extensions to be non-zero, one should assume $k \leq 2j - 2$. Note that whenever $j \leq k \leq 2j-2$ the {\it full} moduli of splitting type $j$ bundles on $Z_k$ is supported on $\ell^{(1)}$. For general $k$, the moduli spaces of bundles on $\ell^{(1)}$ are dense open subspaces of the full moduli spaces of bundles on $Z_k$, so even by restricting to bundles on $\ell^{(1)}$ one obtains a partial description of the moduli space of bundles on all of $Z_k$. We thus refrain from introducing new notation, keeping the same notation as in Not.~\ref{modulinc}.

Moreover, we present the calculation only up to first order in $\hbar$, although we note that the explicit formulae of the star products given in \S \ref{starproducts} enable one to determine the moduli also to higher orders in $\hbar$ (see Rem.~\ref{higherhbar}).

Let $p + p' \hbar$ and $q + q' \hbar$ be two extension classes in $\Ext^1_{\mathcal A} (\mathcal A (j), \mathcal A (-j))$ which are of splitting type $j$, {\it i.e.}\ in canonical $U$-coordinates $p, p', q, q'$ are multiples of $u$.

The bundles defined by $p + p' \hbar$ and $q + q' \hbar$ are isomorphic, if there exist invertible matrices
\begin{equation*}
\begin{pmatrix}
a + a' \hbar & b + b' \hbar \\
c + c' \hbar & d + d' \hbar
\end{pmatrix}
\quad\text{and}\quad
\begin{pmatrix}
\alpha + \alpha' \hbar & \beta  + \beta'  \hbar \\
\gamma + \gamma' \hbar & \delta + \delta' \hbar
\end{pmatrix}
\end{equation*}
whose entries are holomorphic on $U$ and $V$, respectively, such that
\begin{align}
\label{equivalencefirstordermoduli}
\begin{pmatrix}
\alpha + \alpha' \hbar & \beta  + \beta'  \hbar \\
\gamma + \gamma' \hbar & \delta + \delta' \hbar
\end{pmatrix}
\star
\begin{pmatrix}
z^j & q + q' \hbar \\
 0  & z^{-j}
\end{pmatrix}
=
\begin{pmatrix}
z^j & p + p' \hbar \\
 0  & z^{-j}
\end{pmatrix}
\star
\begin{pmatrix}
a + a' \hbar & b + b' \hbar \\
c + c' \hbar & d + d' \hbar
\end{pmatrix}.
\end{align}

We wish to determine the constraints such an isomorphism imposes on the coefficients of $q$ and $q'$. This is more convenient if rewritten by right-multiplying (\ref{equivalencefirstordermoduli}) with the inverse of $\left( \begin{smallmatrix} z^j & q + q' \hbar \\  0  & z^{-j} \end{smallmatrix} \right)$, or more precisely by the right inverse with respect to $\star$, which (modulo $\hbar^2$) is
\[
\begin{pmatrix}
z^{-j} & -q - q' \hbar + 2 z^{-j} \{ z^j, q \} \\
   0   & z^j
\end{pmatrix}.
\]

On $\ell^{(1)}$ we have that $u^2=0$ and therefore $a = a_0 + a_1 u$, $\alpha = \alpha_0 + \alpha_1 u$, etc., where $a_1$, $\alpha_1$, etc.\ are holomorphic functions in $z$.

We first observe that for the classical limit the calculations are given in \cite[\S 3.1]{gasparim}. In particular, following the details of the proof of [{\it ibid.}, Prop.\ 3.3], we may assume that $a_0 = \alpha_0$ are constant, $d_0 = \delta_0$ are constant, and $b = \beta = 0$. Since we already know that on the classical limit the only equivalence on $\ell^{(1)}$ is given projectivization, we may assume that $p = q$ keeping in mind that there is a projectivization to be done in the end. We may also assume that the determinants of the changes of coordinates on the classical limit are $1$. Accordingly, we may simplify (\ref{equivalencefirstordermoduli}) to:
\begin{align}
\label{simplified}
&\begin{pmatrix}
\alpha + \alpha' \hbar & \beta' \hbar \\
\gamma + \gamma' \hbar & \delta + \delta' \hbar
\end{pmatrix} \notag \\
&\qquad
=
\begin{pmatrix}
z^j & p + p' \hbar \\
 0  & z^{-j}
\end{pmatrix}
\star
\begin{pmatrix}
a + a' \hbar & b' \hbar \\
c + c' \hbar & d + d' \hbar
\end{pmatrix}
\star
\begin{pmatrix}
z^{-j} & -p -(q' - 2  \{ z^j, p \}z^{-j}) \hbar \\
 0  & z^j
\end{pmatrix}
\end{align}
where $a_0 = d_0 = \alpha_0 = \delta_0 = 1$.

Since we already know the moduli on the classical limit,
we only need to study the terms containing $\hbar$, which after multiplying are:
\begin{align*}
(1,1) &= a' +\{ z^j a, z^{-j} \} + \{ z^j, a \} z^{-j} +\{ p c, z^{-j} \} + \{ p, c \} z^{-j} + (p c' + p' c) z^{-j} \\
(2,1) &= z^{-2j} c' \\
(1,2) &= -\{ a, p \}z^j - \{ z^j, a \} p + \{ z^j, p \} a + \{ p d, z^j \} + \{ p, d \} + 2 z^{-j} \{ z^j, p \} p c\\
&\quad + z^{2j} b' - (p a' + q' a)z^j+ (p d' + p' d)z^j - (p c' + p' c + q' c) p \\
(2,2) &= d' + \{ z^{-j} d, z^j \} + \{ z^{-j}, d \} z^j - \{ z^{-j} c, p \} - \{ z^{-j}, c \} p- (p c'+ q' c) z^{-j} \\
& \quad + 2 \{ z^j, p \} z^{-2j} c .
\end{align*}

All four terms must be adjusted by using free variables to only contain expressions which are holomorphic on $V$ in order to satisfy (\ref{simplified}). For example, the $(2,1)$ term shows that this condition is satisfied precisely when $c'$ is a section of $\mathcal O (2j)$. A simple verification by computing the Poisson brackets shows that the $(1,1)$ and $(2,2)$ terms can always be adjusted by choosing, say, $c$ and $d'$ appropriately, leaving the coefficients of $a'$ free.

It remains to analyze the term $(1,2)$. Because we are working on the first formal neighbourhood of $\ell$, we may drop terms in $u^2$ (recall that we assume that $p, p', q'$ are multiples of $u$), giving:
\begin{equation}
\label{12term}
\begin{aligned}
(1,2) &= -\{ a, p \} z^j - \{ z^j, a \} p + \{ z^j, p \} a + \{ p d, z^j \} + \{ p, d \} + 2 \{ z^j, p \} p c\\
&\quad + z^{2j} b' - ( pa' + q'a) z^j + (p d' + p' d) z^j .
\end{aligned}
\end{equation}

Since $z^{2j} b'$ is there to cancel out any possible terms having power of $z$ greater or equal to $2j$, we only need to consider the coefficients of the monomials
\begin{equation}
\label{relevantmonomials}
z, z^2, \dotsc, z^{2j-1}
\quad\text{and}\quad
z^{k+1} u, \dotsc, z^{2j-1} u.
\end{equation}

We now calculate moduli spaces for particular values of $k$ and $j$ and for different choices of Poisson structure. We will use the notation $P \in \mathfrak M_j (Z_k)$ to refer to a point in the moduli space, which we write in canonical coordinates as
\begin{equation}\label{point}
P = [p_{1,k-j+1} : p_{1,k-j+2} : \dotsb : p_{1,j-1}]
\end{equation}
so that $P$ corresponds to the isomorphism class of the bundle defined by
\[
\begin{pmatrix}
z^j & p \\
 0  & z^{-j}
\end{pmatrix}
\]
for $p = p_{1,k-j+1} z^{k-j+1} u + \dotsb + p_{1,j-1} z^{j-1} u$.

\begin{example}[$k = 1$ and $j = 2$ and $\sigma_0$] \label{p1}
Consider the Moyal product on $U \subset Z_1$, which by Prop.~\ref{extends} extends to all of $Z_1$.
We calculate the moduli space of vector bundles of splitting type $2$ on $\ell^{(1)}$.

Recall that for these values of $k$ and $j$ the terms appearing in (\ref{12term}) take the following form:
\begin{flalign*}
&& a &= 1 + a_1 (z) u & p &= (p_{10} + p_{11} z) u && \\
&& c &= \textstyle\sum_{l=0}^4 c_{0,l} z^l + \big( \sum_{l=0}^5 c_{1,l} z^l \big) u & p' &= (p'_{10} + p'_{11} z) u && \\
&& d &= 1 + d_1 (z) u & q' &= (q'_{10} + p'_{11} z) u. &&
\end{flalign*}

After calculating the Poisson brackets appearing in (\ref{12term}), one finds that the coefficients of $z$, $z^2$, $z^3$ vanish. Thus, we only need to analyze the coefficients of $z^2u$ and $z^3u$. The required vanishing of the coefficient of $z^2 u$ imposes the condition:
\begin{equation}
\label{q10sigma0}
\begin{aligned}
p'_{10} - q'_{10} &= (a'_{00} -d'_{00} + a_{11} + 5 d_{11}) p_{10} - (a_{10} - 3 d_{11}) p_{11} \\
&\qquad - 4 (c_{01} p_{11}^2 + 2 c_{02} p_{10} p_{11} + c_{04} p_{10}^2).
\end{aligned}
\end{equation}
Similarly, the vanishing of the coefficient of $z^3 u$ imposes:
\begin{equation}
\label{q11sigma0}
\begin{aligned}
p'_{11} - q'_{11} &= (a'_{01} - d'_{01} + 4 a_{12}) p_{10} - (-a'_{00} + d'_{00} - d_{11}) p_{11} \\
&\qquad - 4 (c_{02} p_{11}^2 + 2 c_{03} p_{10} p_{11} + c_{04} p_{10}^2).
\end{aligned}
\end{equation}
Since for $P \in \mathfrak M_2 (Z_1)$ its coefficients $p_{10}$ and $p_{11}$ do not vanish simultaneously, we can choose coefficients of $a, d$ or $c$ to solve both (\ref{q10sigma0}) and (\ref{q11sigma0}). We observe that two out of $a, d, c$ will already have been fixed on a previous step, when canceling coefficients in the $(1,1)$ and $(2,2)$ terms, but there remains always one of them free to be chosen. 
Hence, we can solve both equations for any values of $q'_{10}$ and $q'_{11}$, so that for Poisson structure $\sigma_0 = (1, -\xi)$ two bundles over $\mathcal Z_1 (\sigma_0)$ are isomorphic whenever their classical limits are, giving an isomorphism of moduli spaces
$$
\mathfrak M_2^{(1)} (\mathcal Z_1 (\sigma_0)) \simeq \mathfrak M_2 (Z_1).
$$
\end{example}

\begin{example}[$k = 1$ and $j = 2$ and $\sigma = u$]
\label{m2u}
For $\sigma = (u, -\xi^2 v)$, which by Prop.~\ref{extends} defines a global star product on $Z_1$, equations (\ref{q10sigma0}) and (\ref{q11sigma0}) simplify to
\begin{align*}
p'_{10} - q'_{10} &= (a'_{00} - d'_{00}) p_{10} \\
p'_{11} - q'_{11} &= (a'_{01} - d'_{01}) p_{10} + (a'_{00} - d'_{00}) p_{11}
\end{align*}
which can be written as the Toeplitz system
\begin{equation}
\label{toeplitzj2}
\begin{pmatrix}
p'_{11}- q'_{11} \\
p'_{10}- q'_{10} 
\end{pmatrix}
=
\begin{pmatrix}
p_{10} & p_{11} \\
 0 & p_{10}
\end{pmatrix}
\begin{pmatrix}
a'_{01}- d'_{01} \\
a'_{00}- d'_{00} 
\end{pmatrix}.
\end{equation}
It then follows that the moduli space behaviour is quite different from the case studied in Ex.~\ref{p1}.
Here, if $p_{10}\neq 0$ the matrix $\left( \begin{smallmatrix} p_{10} & p_{11} \\ 0 & p_{10} \end{smallmatrix} \right)$ is invertible, so we can solve (\ref{toeplitzj2}) for any $q'$, giving the equivalence relation
\begin{flalign*}
&& (p_{10},p_{11}, p'_{10},p'_{11}) \sim (\lambda p_{10},\lambda p_{11}, q'_{10},q'_{11}) && \mathllap{\text{if } p_{10}\neq 0.}
\end{flalign*}
So that the fibre of the projection 
$$
\mathfrak M_2^{(1)} (\mathcal Z_1 (\sigma)) \stackrel{\pi}{\longrightarrow} \mathfrak M_2 (Z_1).
$$
to the classical limit is just a point provided $p_{10}\neq 0$.

However, if $p_{10}=0$, that is, over the single point $P=[0:1]$ in the classical moduli space $\mathfrak M_2(Z_1)$, (\ref{toeplitzj2}) imposes the additional constraint that $q'_{10} = p'_{10}$. (The coefficient $q'_{11}$ remains arbitrary
because $p_{11}$ must be nonzero in this case.)

We thus obtain the following equivalence relation:
\begin{flalign*}
&& (p_{10}, p_{11}, p'_{10}, p'_{11}) \sim (\lambda p_{10}, \lambda p_{11}, p'_{10}, q'_{11}) && \mathllap{\text{if } p_{10} = 0.}
\end{flalign*}
Here we observe that both $p'_{11}$ and $q'_{11}$ are arbitrary, so that the resulting equivalence relation on the fibre over $P = [0:1] \in \mathfrak M_2 (Z_1)$
can equivalently be represented by 
$$
(0, p_{11}, p'_{10}, *\hair) \sim (0, \lambda p_{11}, p'_{10}, *\hair)
$$
(where $*$ denotes an arbitrary complex number) is parametrized by the different values of $p'_{10} \in \mathbb C$. Therefore, we have obtained that the fibre $L=\pi^{-1}([0:1])= [0:1]\times \mathbb C$ is an affine line.
Equivalently, the collection of points $L \subset \mathfrak M_2^{(1)} (\mathcal Z_1(\sigma))$ that have $P$ as its classical limit is $L = [0:1] \times \mathbb C$. We can thus view
\begin{equation*}
\begin{tikzpicture}[baseline=-.4em]
\matrix (m) [matrix of math nodes, row sep=1.25em, inner sep=2pt,
column sep=0, ampersand replacement=\&]
{\mathfrak M_2^{(1)}(\mathcal Z_1(\sigma)) \\ \mathfrak M_2( Z_1) \\};
\path[-stealth,line width=.6pt,font=\scriptsize]
(m-1-1) edge (m-2-1)
;
\end{tikzpicture}
\end{equation*}
as the étale space of a skyscraper sheaf supported at $P$.
\end{example}

\begin{example}[$k = 1$ and $j = 3$ and $\sigma_0 = 1$]
\label{p13}
Here the imposed constraints are that the coefficients of $z^2 u, z^3 u, z^4 u, z^5 u$ must vanish. To illustrate the calculation, we list the first two. The analogues of equations (\ref{q10sigma0}) and (\ref{q11sigma0}), corresponding to the vanishing of the coefficients of $z^2 u$ and $z^3 u$, are
\begin{align*}
p'_{1{-1}} {-}\hair q'_{1{-1}} &= (a'_{00} - d'_{00} + 2 a_{11} + 8 d_{11}) p_{1{-1}} + 6 d_{10} p_{10} \\
&\qquad - 6 \Big( c_{00} (p_{11}^2 + 2 p_{10} p_{12}) + 2 c_{01} (p_{10} p_{11} + p_{1{-1}} p_{12}) \\
&\qquad \phantom{ + 6 \Big(} \quad + c_{02} (p_{10}^2 + 2 p_{1{-1}} p_{11}) + 2 c_{03} p_{1{-1}} p_{10} + c_{04} p_{1{-1}}^2 \Big) \\
p'_{10} - q'_{10} &= (a'_{01} {-}\hair d'_{01} {+}\hair 3 a_{12} {+}\hair 9 d_{12}) p_{1{-1}} + (a'_{00} {-}\hair d'_{00} {+}\hair a_{11} {+}\hair 7 d_{11}) p_{10} - (a_{10} {-}\hair 5 d_{10}) p_{11} \\
&\qquad - 6 \Big( 2 c_{00} p_{11} p_{12} + c_{01} (p_{11}^2 + 2 p_{10} p_{12}) + 2 c_{02} (p_{10} p_{11} + p_{1{-1}} + p_{12}) \\ 
&\qquad \phantom{ + 6 \Big(} \quad + c_{03} (p_{10}^2 + 2 p_{1{-1}} p_{11}) + 2 c_{04} p_{1{-1}} p_{10} + c_{05} p_{1{-1}}^2 \Big) 
\end{align*}
and similarly for the coefficients of $z^4u$ and $z^5u$. It then turns out that over the points belonging to the moduli space $\mathfrak M_3 (Z_1)$ we can solve all the constraint equations, so that $q$ is arbitrary, and once again as in Ex.~\ref{p1} we obtain an isomorphism between the quantum and the classical moduli:
$$
\mathfrak M_3^{(1)}(\mathcal Z_1(\sigma_0)) \simeq \mathfrak M_3 (Z_1).
$$
\end{example}

Generalizing Exs.~\ref{p1} and \ref{p13} we obtain the following theorem.

\begin{theorem}\label{iso}
Let $\sigma_0$ be the Poisson structure which in canonical coordinates is $\sigma_0=(1, -\xi)$. Then for each $j$ we obtain an isomorphism
\[
\mathfrak M_j^{(1)} (\mathcal Z_1(\sigma_0)) \simeq \mathfrak M_j (Z_1).
\]
\end{theorem}

\begin{example}[$k = 1$ and $j = 3$ and $\sigma = u$]
\label{j3u}
For $j = 3$, the relevant monomials of (\ref{12term}) are $z^2 u, \dotsc, z^5 u$ and
the vanishing condition can be written as the Toeplitz system
\begin{equation}
\label{toeplitzj3}
\begin{pmatrix}
p'_{12}- q'_{12} \\
p'_{11}- q'_{11}  \\
p'_{10}- q'_{10} \\
p'_{1-1}-q'_{1-1} 
\end{pmatrix}
=
\begin{pmatrix}
 p_{1-1} & p_{10} & p_{11} & p_{12}  \\
0 & p_{1-1} & p_{10} & p_{11} \\
 0 & 0 &  p_{1-1} & p_{10}\\
 0 & 0 & 0    &      p_{1-1} 
\end{pmatrix}
\begin{pmatrix}
a'_{03}- d'_{03}  \\
a'_{02}- d'_{02} \\
a'_{01}- d'_{01}  \\
a'_{00}- d'_{00} 
\end{pmatrix}.
\end{equation}
As in Ex.~\ref{m2u} the fibres of the projection 
\begin{equation*}
\begin{tikzpicture}[baseline=-.4em]
\matrix (m) [matrix of math nodes, row sep=1.5em, inner sep=2pt,
column sep=0, ampersand replacement=\&]
{
\mathfrak M_3^{(1)} (\mathcal Z_1 (\sigma)) \\
\mathfrak M_3 (Z_1) \\
};
\path[-stealth,line width=.6pt,font=\scriptsize]
(m-1-1) edge node[pos=.45,left=-.3ex] {$\pi$} (m-2-1)
;
\end{tikzpicture}
\end{equation*}
vary depending on the coordinates of the point $[p_{1{-1}}:p_{10}:p_{11}:p_{12}] \in \mathfrak M_3 (Z_1) \subset \mathbb P^3$.
We have:
\begin{itemize}
\item If $p_{1-1}\neq 0$ we can solve all the equations by choosing $a'$ and $d'$ appropriately. Hence, 
there exists an isomorphism for any value of $q'$ and consequently the fibre in this case is only a point.
Thus 
$$
S_0 := \{P \in \mathfrak M_3 (Z_1) \mid p_{1-1} \neq 0 \}
$$
is an open set of the classical moduli space over which the map $\pi$ is an isomorphism.

\item Now assume $p_{1-1}= 0$ but $p_{10} \neq 0$. Then the fourth equation imposes the condition 
$q'_{1-1}=p'_{1-1}$ but the other equations can be solved for any values of $q'_{10},q'_{11},q'_{12}$. So that isomorphism gives the equivalence relation
\[
[0:1:p_{11}:p_{12}] (p'_{1-1},p'_{10},p'_{11},p'_{12}) \sim [0: 1:p_{11}:p_{12}](p'_{1-1},*\;,*\;,*\hair)
\]
where $*$ stands for an arbitrary complex value. Thus, the fibre of $\pi$ over a point $[0:1:p_{11}:p_{12}]$ is a copy of $\mathbb C$ parametrized by the different values of $p'_{1-1}$. In other words, setting 
$$
S_1 := \{P \in  \mathfrak M_3 (Z_1) \mid p_{1-1}= 0, \; p_{10} \neq 0 \}
$$
we have that if $P\in S_1$ then $\pi^{-1}(P) \simeq \mathbb C$.

\item Next assume that $p_{1-1}= p_{10}=0$ but $p_{11}\neq 0$, then the third and fourth equations of (\ref{toeplitzj3}) impose the conditions $q'_{1-1}=p'_{1-1}$ and $q'_{10}=p'_{10}$ but the remaining two equations can be solved for any values of $q'_{11},q'_{12}$. Hence in this case, isomorphism imposes the following equivalence relation
\[
[0:0:1:p_{12}](p'_{1-1},p'_{10},p'_{11},p'_{12}) \sim [0:0:1:p_{12}](p'_{1-1},p'_{10},*\,,*)
\]
with the fibre of $\pi$ over such a point being a copy of $\mathbb C^2$ parametrized by the values of $(p'_{1-1}, p'_{10})$.
Thus, setting 
$$
S_2 := \{P \in \mathfrak M_3 (Z_1) \mid p_{1-1} = p_{10}=0, \, p_{11} \neq 0 \}
$$
we have that if $P \in S_2$ then $\pi^{-1} (P) \simeq \mathbb C^2$.
\item Lastly there would be the point $[0:0:0:1]$ to be considered, but it does not belong to $\mathfrak M_3 (Z_1)$ because it corresponds to a vector bundle with charge $5$ (see Def.~\ref{charge} and Table \ref{table}), so we are already done. 
\end{itemize}
In conclusion,
\begin{equation*}
\begin{tikzpicture}[baseline=-.4em]
\matrix (m) [matrix of math nodes, row sep=1.5em, inner sep=2pt,
column sep=0, ampersand replacement=\&]
{
\mathfrak M_3^{(1)} (\mathcal Z_1 (\sigma)) \\
\mathfrak M_3 (Z_1) \\
};
\path[-stealth,line width=.6pt,font=\scriptsize]
(m-1-1) edge node[pos=.45,left=-.3ex] {$\pi$} (m-2-1)
;
\end{tikzpicture}
\end{equation*}
can be viewed as the étale space of a constructible sheaf, with stalks of dimension $i$ over the strata $S_i$ for $i=0,1,2$.
\end{example}

\begin{example}[$\sigma_U$ multiple of $u$, $k$ and $j$ arbitrary]
\label{ujarb}
Let $\sigma$ be a holomorphic Poisson structure on $Z_k$ such that $\sigma_U$ is a multiple of $u$ and let $\mathcal Z_k (\sigma)$ be any deformation quantization of $Z_k$. (Note that by Lem.~\ref{poissonstructures} any Poisson structure on $Z_{\geq 3}$ is of this form.)

Then the moduli space $\mathfrak M_j^{(1)} (Z_k{(\sigma)})$ may be described by the Toeplitz system: 
{\small
\[
\left(
\begin{tikzpicture}[x=4.7em,y=2.75em,baseline=-8.5em]
\node[inner sep=0pt,shape=rectangle,baseline=(11.base)] (11) at (1,-1) {$p'_{1,j-1} - q'_{1,j-1}$\strut};
\node[inner sep=1pt,shape=rectangle,baseline=(11.base)] (21) at (1,-2) {$p'_{1,j-2} - q'_{1,j-2}$\strut};
\node[inner sep=1pt,shape=rectangle,baseline=(11.base)] (41) at (1,-4) {$p'_{1,k-j+2} {-} q'_{1,k-j+2}$\strut};
\node[inner sep=0pt,shape=rectangle,baseline=(11.base)] (51) at (1,-5) {$p'_{1,k-j+1} {-} q'_{1,k-j+1}$\strut};
\path[dash pattern=on 0pt off 5pt, line width=1pt, line cap=round] (21) edge (41);
\end{tikzpicture}
\right)
=
\left(
\begin{tikzpicture}[x=4.1em,y=2.75em,baseline=-8.5em]
\node[inner sep=0pt,shape=rectangle,baseline=(11.base)] (11) at (1,-1) {$p_{1,k-j+1}$\strut};
\node[inner sep=0pt,shape=rectangle,baseline=(22.base)] (22) at (2,-2) {$p_{1,k-j+1}$\strut};
\node[inner sep=1pt,shape=rectangle,baseline=(33.base)] (33) at (3,-3) {$p_{1,k-j+1}$\strut};
\node[inner sep=0pt,shape=rectangle,baseline=(55.base)] (55) at (5,-5) {$p_{1,k-j+1}$\strut};
\node[inner sep=0pt,shape=rectangle,baseline=(21.base)] (12) at (2,-1) {$p_{1,k-j+2}$\strut};
\node[inner sep=4pt,shape=rectangle,baseline=(44.base)] (13) at (3,-1) {$p_{1,k-j+3}$\strut};
\node[inner sep=0pt,shape=rectangle,baseline=(44.base)] (23) at (3,-2) {$p_{1,k-j+2}$\strut};
\node[inner sep=0pt,shape=rectangle,baseline=(44.base)] (45) at (5,-4) {$p_{1,k-j+2}$\strut};
\node[inner sep=0pt,shape=rectangle,baseline=(44.base)] (35) at (5,-3) {$p_{1,k-j+3}$\strut};
\node[inner sep=0pt,shape=rectangle,baseline=(44.base)] (15) at (5,-1) {$p_{1,j-1}$\strut};
\node[inner sep=.5pt,shape=circle,baseline=(11.base)] (21) at (1,-2) {$0$\strut};
\node[inner sep=.5pt,shape=circle,baseline=(33.base)] (31) at (1,-3) {$0$\strut};
\node[inner sep=.5pt,shape=circle,baseline=(33.base)] (32) at (2,-3) {$0$\strut};
\node[inner sep=.5pt,shape=circle,baseline=(55.base)] (51) at (1,-5) {$0$\strut};
\node[inner sep=.5pt,shape=circle,baseline=(55.base)] (53) at (3,-5) {$0$\strut};
\node[inner sep=.5pt,shape=circle,baseline=(55.base)] (54) at (4,-5) {$0$\strut};
\path[dash pattern=on 0pt off 5pt, line width=1pt, line cap=round] (33) edge (55);
\path[dash pattern=on 0pt off 5pt, line width=1pt, line cap=round] (23) edge (45);
\path[dash pattern=on 0pt off 5pt, line width=1pt, line cap=round] (13) edge (35);
\path[dash pattern=on 0pt off 5pt, line width=1pt, line cap=round] (13) edge (15);
\path[dash pattern=on 0pt off 5pt, line width=1pt, line cap=round] (15) edge (35);
\path[dash pattern=on 0pt off 5pt, line width=1pt, line cap=round] (32) edge (54);
\path[dash pattern=on 0pt off 5pt, line width=1pt, line cap=round] (31) edge (53);
\path[dash pattern=on 0pt off 5pt, line width=1pt, line cap=round] (31) edge (51);
\path[dash pattern=on 0pt off 5pt, line width=1pt, line cap=round] (51) edge (53);
\end{tikzpicture}
\right)
\left(
\begin{tikzpicture}[x=4.7em,y=2.75em,baseline=-8.5em]
\node[inner sep=0pt,shape=rectangle,baseline=(11.base)] (11) at (1,-1) {$a'_{0,2j-k-2} {-} d'_{0,2j-k-2}$\strut};
\node[inner sep=1pt,shape=rectangle,baseline=(11.base)] (21) at (1,-2) {$a'_{0,2j-k-3} {-} d'_{0,2j-k-3}$\strut};
\node[inner sep=1pt,shape=rectangle,baseline=(11.base)] (41) at (1,-4) {$a'_{0,1} - d'_{0,1}$\strut};
\node[inner sep=0pt,shape=rectangle,baseline=(11.base)] (51) at (1,-5) {$a'_{0,0} - d'_{0,0}$\strut};
\path[dash pattern=on 0pt off 5pt, line width=1pt, line cap=round] (21) edge (41);
\end{tikzpicture}
\right)
\]
}
\end{example}

Exs.~\ref{j3u} and \ref{ujarb} readily generalize to give:

\begin{theorem}
\label{noniso}
The quantum moduli space $\mathfrak M_j^{(1)} (\mathcal Z_k (\sigma))$ can be viewed as the étale space of a constructible sheaf over the classical moduli space $\mathfrak M_j (Z_k)$, the sheaf being trivial over the open set
$$
S_0:= \{P \in \mathfrak M_j (Z_k) \mid p_{1,k-j+1}  \neq 0 \}
$$
and with stalk of dimension $i$ over the locally closed subvarieties
$$
S_i := \{P \in \mathfrak M_j (Z_k) \mid p_{1,k-j+1} = \cdots = p_{1,k-j+i} = 0,   \, p_{1,k-j+1+i} \neq 0 \}.
$$

\end{theorem}

\begin{proof} The classical moduli space gets stratified into disjoint subsets $S_i$ over which the fibre of the projection 
\begin{equation*}
\begin{tikzpicture}[baseline=-.4em]
\matrix (m) [matrix of math nodes, row sep=1.5em, inner sep=2pt,
column sep=0, ampersand replacement=\&]
{
\mathfrak M_j^{(1)} (\mathcal Z_k (\sigma)) \\
\mathfrak M_j (Z_k) \\
};
\path[-stealth,line width=.6pt,font=\scriptsize]
(m-1-1) edge node[pos=.45,left=-.3ex] {$\pi$} (m-2-1)
;
\end{tikzpicture}
\end{equation*}
has dimension $i$. 
The strata are the dense open subset
$$
S_0:= \{P \in \mathfrak M_j (Z_k) \mid p_{1,k-j+1}  \neq 0 \}
$$
and locally closed subsets of decreasing dimension
\begin{flalign*}
&& S_i := \{P \in \mathfrak M_j (Z_k) \mid p_{1,k-j+1} = \cdots = p_{1,k-j+i} = 0,   \, p_{1,k-j+1+i} \neq 0 \}. && \qedhere
\end{flalign*}
\end{proof}

\begin{remark}[Higher powers of $\hbar$]
\label{higherhbar}
Repeating the calculation for higher powers of $\hbar$, one finds that $\mathfrak M_2^{(2)} (\mathcal Z_1 (\sigma)) \simeq \mathfrak M_2^{(1)} (\mathcal Z_1 (\sigma))$
so that for splitting type $2$, considering terms up to $\hbar^2$ does not give any more vector bundles. We thus expect that for arbitrary powers of $\hbar$ one obtains an isomorphism $\mathfrak M_2 (\mathcal Z_1 (\sigma)) \simeq \mathfrak M_2^{(1)} (\mathcal Z_1 (\sigma))$ with all splitting type $2$ vector bundles supported on the first neighbourhood of $\hbar$. In particular, this implies that bundles of splitting type $2$ on $\mathcal Z_1 (\sigma)$ are {\it algebraic} in the sense of Def.~\ref{formally}.

Similarly, Thm.~\ref{iso} would imply
\[
\mathfrak M_j (\mathcal Z_1 (\sigma_0)) \simeq \mathfrak M_j (Z_1)
\]
for the minimally degenerate Poisson structure $\sigma_0$ and arbitrary $j$, which corresponds to the statement that bundles of arbitrary splitting type on $\mathcal Z_1 (\sigma_0)$ are algebraic.

On the other hand, the calculation of $\mathfrak M_2^{(2)} (\mathcal Z_1 (\sigma))$ also suggests that for $j \geq 3$ and Poisson structures which are degenerate on all of $\ell$, for example any Poisson structure on $Z_{k \geq 3}$, the moduli space $\mathfrak M_j (\mathcal Z_k (\sigma))$ should also contain bundles whose presentation requires nontrivial coefficients of $\hbar^n$ for $n > 1$. In particular, this implies that the moduli space $\mathfrak M_j^{(n)} (\mathcal Z_k (\sigma))$ is larger than $\mathfrak M_j^{(1)} (\mathcal Z_k (\sigma))$.

The calculation of $\mathfrak M_2^{(2)} (\mathcal Z_1 (\sigma))$ is analogous to the calculations of Exs.~\ref{p1} and \ref{m2u} and can be reproduced by using the Kontsevich star product extended to $Z_k$ as given in Prop.~\ref{extends}, whose terms up to second order can be obtained from Lem.~\ref{secondorder}. However, the calculation is considerably longer and we do not include it here, leaving a more general investigation for future work.
\end{remark}

\section{Classical instantons}
\label{classicalinstantons}

By definition an {\it instanton} on a 4-manifold
$X$ is a connection $A$ minimizing the Yang--Mills functional 
$$
\mathrm{YM}(A) = \int_X \tr F \wedge F
$$ 
where $F$ is the curvature of $A$. The Euler--Lagrange equations for the 
Yang--Mills functional 
\begin{equation}
\label{ym}
D (\ast F) = 0
\end{equation}
are called the {\it Yang--Mills equations} (here $\ast$ denotes Hodge dual).
Thus an instanton is a solution of the Yang--Mills equations. 
A linearized version of these equations is given by the anti-self-duality (ASD) equations
\begin{equation}
\label{asd}
F^+ = \frac12(F + \ast F) = 0.
\end{equation} 
Equation (\ref{asd}) is sometimes called the {\it instanton equation} because its solutions also satisfy (\ref{ym}).

The translation from gauge  theory to complex geometry is made via the so-called 
Kobayashi--Hitchin correspondence, which relates instantons and vector bundles as well 
as their moduli spaces. If $X$ is a complex K\"ahler surface and $E$ is an 
$\SU(2)$ bundle on $X$, then the moduli space $M_E$ of irreducible ASD connections on $E$ is a complex analytic space and each point in $M_E$ has a neighbourhood which is the base of a universal deformation of the corresponding stable vector bundle 
\cite[Prop.\ 6.4.4]{DK}.
A version of such a correspondence for noncompact surfaces states:

\begin{lemma}\label{nk}\cite[Cor.\ 5.5]{GKM}
A holomorphic $\SL (2, \mathbb C)$ vector bundle on $Z_k$ corresponds to an $\SU(2)$ instanton if and only if its splitting type is a multiple of $k$. 
 \end{lemma}

The surfaces $Z_k$ have rich moduli spaces of instantons, which unfortunately disappear under any small commutative 
deformation of $Z_k$. In fact, for $j \equiv 0 \mod k$, we have:

\begin{theorem}\cite[Thm.\ 4.11]{BGK2}
\label{existence}
The moduli space of irreducible $\SU (2)$ instantons on $Z_k$ with charge (and splitting type) $j$ is a quasi-projective variety of dimension $2j-k-2$. 
\end{theorem}

In contrast:

\begin{theorem}\cite[Thm.\ 7.3]{BrG}
\label{empty}
Let $\mathcal Z_k (\tau)$ be a nontrivial commutative deformation of $Z_k$. Then the moduli spaces of irreducible $\SU (2)$ instantons on $\mathcal Z_k (\tau)$ are empty.
\end{theorem}

Disappearance of instantons under classical deformations gave us a strong  motivation to explore noncommutative directions of deformations. From the point of view of algebraic deformation theory, both classical and noncommutative deformations can be regarded on equal footing --- as components in the Hochschild--Kostant--Rosenberg decomposition of the Hochschild cohomology
\[
\HH^2 (Z_k) \simeq \H^1 (Z_k, \mathcal T_{Z_k}) \oplus \H^0 (Z_k, \Wedge^2 \mathcal T_{Z_k})
\]
which parametrizes deformations of the Abelian of category of (quasi)co\-he\-rent sheaves in the sense of \cite{lowenvandenbergh}. (Note that simultaneous deformations of $Z_k$ in commutative and noncommutative directions may be obstructed. The obstruction calculus was studied in \cite{barmeierfregier}.)
However, from a physics point of view we hoped, and intuitively expected, that moduli of instantons appear again on directions of noncommutative deformations. We will see that this is indeed the case.
Furthermore, we discover that some instantons react wildly to certain types of quantization, causing quantum  moduli spaces to become larger than the classical ones. 

The geometry underlying the disappearance of instantons on commutative deformations 
is the fact that if $\mathcal Z_k (\tau)$ is any nontrivial commutative deformation of $Z_k$, then every holomorphic vector bundle on $\mathcal Z_k (\tau)$ splits as a direct sum of line bundles \cite[Thm.\ 6.10]{BrG}.
The key issue here is that for $\tau\neq 0$ the deformations $\mathcal Z_k (\tau)$ are affine \cite[Thm.\ 6.18]{BrG}. In particular, by \cite[Thm.\ 6.6]{BrG} such a nontrivial deformation contains no compact complex curves. In physics language, we may say that there is no compact manifold which can hold the instanton charge. 
 
In Yang--Mills theory instantons are well known to carry topological charges. Under the Kobayashi--Hitchin correspondence the charge of an $\SU (2)$ instanton on a complex surface translates into the second Chern class of its corresponding $\SL (2, \mathbb C)$ holomorphic vector bundle. Here we use instead the concept of {\it local charge}, given that second Chern class of a bundle on $Z_k$ vanishes, since $\H^4 (Z_k, \mathbb Z) \simeq \H^4 (S^2, \mathbb Z) = 0$. The terminology ``local charge'' is motivated by the fact that it provides a local contribution to the second Chern class when we consider a compact surface containing an embedded $Z_k$.

\begin{definition}\label{inscharge}
We define the {\it normalized charge} of an instanton on $Z_k$ to be the sum of the local second Chern class (Def.~\ref{charge}) of the bundle to which it corresponds and $\epsilon$, where $\epsilon = 1$ if $k \geq 2$ and $0$ when $k = 1$. In what follows we will simply refer to this normalized charge as the {\it charge} of the instanton for brevity. 
\end{definition}

\begin{remark}
The reason for this normalization is that, once normalized, the minimal charge of an $\SU (2)$ instanton with splitting type $j$ equals $j$, and this allows us to express the theorems that follow in a much simpler way. We observe that the addition of $\epsilon = 1$ in the cases $k \geq 2$ corresponds to the fact that the surface $X_k$ obtained by contracting the $\mathbb P^1$ to a point has one singularity. In case we considered more general surfaces containing other contractible curves, the correct normalization should most likely be to count all singularities obtained by their contraction. For bounds on the values of instanton charges see \cite{BGK2} and \cite{GKM}.

We note also that there exist more than one notion of Chern classes for sheaves on singular varieties. A particularly useful one, presented by Blache \cite{Bl}, is the concept of orbifold Chern class, defined using the familiar integration formula $\int_X \mathrm{ch} (E) \, \mathrm{td} (X)$ appearing in Hirzebruch--Riemann--Roch formulas, but which in the case of singular varieties differs from the Euler characteristic by 
a weighted counting of singularities. In fact, Blache proves the following formula:
$$
\chi(X, E) = \int_X \mathrm{ch} (E) \,\mathrm{td} (X) + \sum_{x \in \operatorname{Sing} (X)} \mu_{x,X} (E).
$$
Here our $\epsilon$ may be regarded as keeping track of when the defect term $\mu$ is nonzero. 
\end{remark}

As an illustration we give instanton invariants from \cite{BG} in Table \ref{table}.
\begin{table}[ht]	
\begin{center}
\begingroup
\renewcommand*{\arraystretch}{1.15}
\begin{tabular}{c||c|c|c}
monomial  & \eqmakebox[table]{width} & \eqmakebox[table]{height} & \eqmakebox[table]{charge} \\
\hline
$z^{-1}u$ & 3 & 2 & 5\\
$u$ & 1 & 2 & 3\\
$zu$ & 1 & 2 & 3\\
$z^{2}u$ & 3 & 2 & 5\\
\hline
$u^2$ & 3 & 3 & 6\\
$zu^2$ & 2 & 3 & 5\\
$z^{2}u^2$ & 3 & 3 & 6\\
\hline
$zu^3$ & 5 & 3 & 7\\
$z^{2}u^3$ & 5 & 3 & 7\\
\hline
$z^{2}u^4$ & 5 & 3 & 8\\
\hline
zero & 6 &3 &  9\\
\end{tabular}
\endgroup
\vspace{2mm}
\end{center}
\caption{Numerical invariants for rank $2$ vector bundles of splitting type $j = 3$ on $Z_1$}
\label{table}
\end{table}

\section{Rebel instantons}
\label{instantons}

We wish to investigate the effect that noncommutative deformations have on instantons --- these effects can already be observed at first order in $\hbar$. We thus reinterpret of the results of \S\ref{vectorbundles} into the language of instantons.

Noncommutative versions of (\ref{ym}) and (\ref{asd}) and instanton solutions for the cases of noncommutative $\mathbb R^4$ are presented in \cite{NS}, with self-duality being given by the generalized ASD equations
$$
F^+_{\mu\nu}=0
$$
where the curvature of the connection is calculated by the generalized formula
$$
F^i_{\mu\nu, j} = \partial_\mu A^i_{\nu,j}- \partial_\nu A^i_{\mu,j}+A^i_{\mu,k}\star A^k_{\nu,j} - A^i_{\nu,k} \star A^k_{\mu,j}.
$$
ASD connections are then automatically solutions to the deformed Yang--Mills equations:
$$
\partial_\mu F_{\mu\nu} - A_\mu\star F_{\mu\nu} = 0.
$$
Instantons on noncommutative $\mathbb R^4$ in this sense were further studied in \cite{SW} for their relations with string theory, and in \cite{KKO} on noncommutative projective planes. For our noncommutative deformations of the $Z_k$, we have obtained global star products and one could also approach the study of instantons using such deformed equations. However, it is more convenient to work directly in the language of vector bundles.

\begin{definition}
Based on Def.~\ref{split} and the result of Lem.~\ref{nk}, a formally algebraic bundle on $\mathcal Z_k (\sigma)$ of rank $2$ and splitting type $nk$ is called a {\it formal instanton}. The {\it charge} of a formal instanton on $\mathcal Z_k (\sigma)$ is defined as the charge of its classical limit (Def.~\ref{inscharge}), {i.e.}\ obtained by setting $\hbar = 0$.
\end{definition}

\begin{definition}
\label{insrep}
According to Def.~\ref{canonical}, instantons on $Z_k$ and on deformations of $Z_k$ can be represented in canonical coordinate charts by a pair $(j, \mathbf p)$ of an integer and a formal expression $\mathbf p = \sum p_n \hbar^n$ where $p_n$ are polynomials. For a fixed splitting type $j$, the instanton is thus determined by the coefficients of the corresponding polynomials, and in this case we denote the point on the corresponding moduli space simply by $P$ as in (\ref{point}).
\end{definition}

\begin{notation}
We will denote by $\mathbb{QI}_j (\mathcal Z_k (\sigma))$ the moduli space of formal instantons of charge $j$ on $\mathcal Z_k (\sigma)$, and by $\mathbb{QI}_j^{(n)}(\mathcal Z_k (\sigma))$ the subspace obtained by imposing the cut $\hbar^{n+1} = 0$. Hence $\mathbb{QI}^{(0)}_j (\mathcal Z_k (\sigma))$ equals the classical moduli space $\mathbb{MI}_j (Z_k)$ of instantons of charge $j$. 
\end{notation}

Therefore, by definition we have:

\begin{lemma}
The classical limit of any formal instanton on $\mathcal Z_k (\sigma)$ is an instanton on $Z_k$.
\end{lemma}

\begin{proof} The correspondence between vector bundles and instantons may be applied to both quantum and classical moduli. The projection onto classical moduli translates as
\begin{equation*}
\begin{tikzpicture}[baseline=-.4em]
\matrix (m) [matrix of math nodes, row sep=.5em, inner sep=2pt,
column sep=3em, ampersand replacement=\&]
{
\mathfrak M_j^{(1)} (\mathcal Z_k (\sigma)) \&\& \mathbb{QI}_j^{(1)} (\mathcal Z_k (\sigma)) \\
                       \& \longleftrightarrow \& \\
\mathfrak M_j (Z_k)                     \&\& \mathbb{MI}_j (Z_k) \\
};
\path[-stealth,line width=.6pt,font=\scriptsize]
(m-1-1) edge node[pos=.45,left=-.3ex] {$\pi$} (m-3-1)
(m-1-3) edge node[pos=.45,left=-.3ex] {$\pi$} (m-3-3)
;
\end{tikzpicture}
\end{equation*}
\end{proof}

We now interpret the statements of Thms.~\ref{iso} and \ref{noniso} in terms of instantons. These get rephrased as:

\begin{theorem}
\label{quantumclassical}
Let $\sigma_0$ be the Poisson structure on $Z_1$ which in canonical coordinates is 
described by $\sigma_0 = (1, -\xi)$,
hence $\sigma_0$ is degenerate at a single point of the line $\ell \subset Z_1$.
Then for each value of the charge $j$ we obtain an isomorphism between the quantum and the classical instanton moduli spaces
\[
\mathbb{QI}_j^{(1)} (\mathcal Z_1 (\sigma_0)) \simeq \mathbb {MI}_j (Z_1).
\]
\end{theorem}

\begin{theorem}\label{qi}
The quantum instanton moduli space $\mathbb{QI}_j^{(1)} (\mathcal Z_k (\sigma))$ can be viewed as the étale space of a constructible sheaf over the classical instanton moduli space $\mathbb {MI}_j( Z_k)$ which is supported on a closed subvariety, being trivial over 
$$
S_0:= \{P \in \mathbb{MI}_j (Z_k) \mid p_{1,k-j+1} \neq 0 \}
$$
and having stalk of dimension $i$ over 
$$
S_i := \{P \in \mathbb {MI}_j (Z_k) \mid p_{1,k-j+1} = \cdots = p_{1,k-j+i} = 0, \; p_{1,k-j+1+i} \neq 0 \}.
$$
\end{theorem}

\begin{notation}
When regarding the quantum moduli space $\mathbb{QI}_j (\mathcal Z_k (\sigma))$ as the étale space of a sheaf over the classical moduli space $\mathbb{MI}_j (Z_k)$, we will use the notation $\mathcal Q^\sigma$ and call it the {\it quantizing sheaf}.
\end{notation}

\begin{definition}
\label{rebel}
A (classical) instanton $A$ is called a {\it rebel instanton} for $\sigma$ if the stalk $\mathcal Q^\sigma_A$ of the quantizing sheaf $\mathcal Q^\sigma$ at $A$ is nontrivial. The dimension of the stalk at $A$ is called the {\it level of rebelliousness} of $A$. Hence, if the stalk  $\mathcal Q^\sigma_A$ has rank $n$ then $A$ is said to present level $n$ rebelliousness, or equivalently, $A$ is called {\it $n$-rebel instanton}. For brevity the vocabulary {\it rebel} will be used to denote rebelliousness of any level $n \geq 1$. 

An instanton that is not rebel for $\sigma$ is called {\it $\sigma$-tame}.
\end{definition}

In particular, in the language of Def.~\ref{rebel} (see also Not.~\ref{minimally}) we can rephrase Thm.~\ref{quantumclassical} as follows.

\begin{theorem}
\label{alltame}
Let $\sigma_0$ be a minimally degenerate Poisson structure on $Z_1$. Then all instantons on $Z_1$ are $\sigma_0$-tame.
\end{theorem}

Note that the zero section $\ell$ of $Z_1$ contracts to a smooth point and Thm.~\ref{alltame} stands in stark contrast to the case of $k \geq 3$ where the zero section $\ell$ of $Z_k$ contracts to a $\frac1k (1,1)$ singularity.

\begin{theorem}
\label{rebels}
Let $k \geq 3$. Then $Z_k$ produces rebel instantons for any holomorphic Poisson structure $\sigma$.
\end{theorem}

To be interesting from a physical perspective, it is important that the first-order deformations parametrized by a holomorphic Poisson structure can be continued to all higher orders, which was shown in \S \ref{quantizableimmersions}, at least for Poisson structures tangent to a fibre of the projection to $\mathbb P^1$. Thms.~\ref{alltame} and \ref{rebels} now exhibit a phenomenon that can already be observed at first order in $\hbar$ and for higher orders in $\hbar$ we expect Thms.~\ref{alltame} and \ref{rebels} to generalize as follows.

Phrased in the language of (rebel) instantons, Rem.~\ref{higherhbar} suggests that for the minimally degenerate Poisson structure $\sigma_0$ on $Z_1$, we expect an isomorphism $\mathbb{QI}_j (\mathcal Z_1 (\sigma_0)) \simeq \mathbb{MI}_j (Z_1)$, {\it i.e.}\ the absence of rebel instantons implies that the moduli of noncommutative and classical instantons are isomorphic for arbitrary powers of $\hbar$.

On the other hand we expect the existence of $n$-rebel instantons (as for example in Thm.~\ref{rebels}) to imply the existence of noncommutative instantons whose minimal representatives $\mathbf p = \sum_{n \geq 0} p_n \hbar^n$ contain non-vanishing terms $p_n$, {\it i.e.}\ $n$-rebel instantons produce noncommutative instantons up to $\hbar^n$.

\appendix
\section{Cohomology}

In this appendix we determine global sections of $Z_k$ with line bundle coefficients. 
These are used in the proof of Lem.~\ref{poissonstructures} to determine the space of Poisson structures on $Z_k$.

The coordinate ring of $Z_k$ is
\[
R = \H^0 (Z_k, \mathcal O) = {\mathbb C}[x_0, x_1, \dotsc, x_k] \bigl/ (x_i x_{j+1} - x_{i+1} x_j)_{0 \leq i < j \leq k-1}.
\]
The contraction map $Z_k \to \Spec R$ is given in $(z,u)$-coordinates by $z^i u \mapsto x_i$. We compute $\H^0({Z_k,\mathcal O}(j))$.

\begin{lemma}
\label{j>0}
The cohomology $\H^0 (Z_k, \mathcal O (j))$ for $j \geq 0$ is generated as an $R$-module by the monomials $\beta_i = z^i$ for $0 \leq i \leq j$ with relations
\[
\beta_i x_{l-1} - \beta_{i-1} x_l = 0
\]
where $1 \leq l \leq k$ and $1 \leq i \leq j$.
\end{lemma}

\begin{proof}
Let $\mathfrak U = \{ U, V \}$ be our canonical coordinates on $Z_k$ and let $\sigma \in \check{\mathrm C}^0 (\mathfrak U, {\mathcal O}(j))$ be a $0$-cochain,
thus $\sigma$ consists of a pair of holomorphic sections $(\sigma_U,\sigma_V)$.
Writing out the general form of such a cochain, we have 
$\sigma_U = \sum_{i=0}^\infty\sum_{l=0}^\infty \sigma_{il} z^l u^i$ as 
this is just an arbitrary holomorphic map on $U$, where the line bundle 
trivializes. As we know, $\H^0$ computes global holomorphic sections, 
so to find the cohomology class of $\sigma$ we just need to find the most general such $\sigma$ that extends holomorphically to $V$. 
The transition function for the line bundle ${\mathcal O}(j)$ is 
$T=z^{-j}$. Changing coordinates then gives 
$T\sigma = z^{-j} \sum_{i=0}^\infty\sum_{l=0}^\infty \sigma_{il}z^lu^i$
which needs to be holomorphic on $V$. This gives the condition 
that $T\sigma$ must contain only terms $z^ru^s$ with 
$r \leq ks$. Terms that do not satisfy this condition are 
not cocycles, and must be removed from the expression of $T\sigma$.
We are thus left with: 
$T\sigma = \sum_{i=0}^\infty\sum_{l=0}^{ki+j} \sigma_{il}z^{l-j}u^i$
or equivalently
$\sigma = \sum_{i=0}^{\infty}\sum_{l=0}^{ki+j} \sigma_{il}z^lu^i$
which can be rewritten as
$$\sigma= \sum_{l=0}^j \sigma_{0l}z^l+
 \sum_{i=1}^{\infty}\sum_{l=0}^{ki+j} \sigma_{il}z^{l-ki}(z^ku)^i 
\text{.}$$
Now notice that on the right-hand side every term on the second sum can be obtained
from a term on the first sum by multiplying by $(z^ku)^i$ which are 
global holomorphic functions on $Z_k$, thus $\H^0(Z_k,{\mathcal O}(j))$
with $j\geq 0$ is generated as an $R$-module by the monomials $\beta_i = z^i$ for $0\leq i\leq j$. Since $j \geq 0$ there is at least one such $\beta_i$. These satisfy the equalities
\begin{flalign*}
&& \beta_1 x_0 = \eqmakebox[beta][r]{$\beta_0 x_1$} &= \eqmakebox[q][r]{$z u$} && \\
&& \beta_2 x_0 = \beta_1 x_1 = \eqmakebox[beta][r]{$\beta_0 x_2$} &= \eqmakebox[q][r]{$z^2 u$} && \\[-.6em]
&& \vdots\hspace{3.8em}\vdots\hspace{.9em} & \hspace{4.1em}\vdots && \\[-.3em]
&& \beta_j x_{k-1} = \eqmakebox[beta][r]{$\beta_{j-1} x_k$} &= \eqmakebox[q][r]{$z^{k+j-1} u	$}. && \qedhere
\end{flalign*}
\end{proof}

To compute $\H^0 (Z_k, \mathcal O (-j))$, set $\nu = -j \mod k$, so that 
$$
-j = -q k + \nu
$$
with $0 \leq \nu < k$. 

\begin{lemma}
\label{j<0}
The cohomology group $\H^0 (Z_k, \mathcal O (-j))$ for $j \geq 0$ is generated by the monomials $\alpha_i = z^i u^q$, for $0 \leq i \leq \nu$, with relations
\[
\alpha_i x_{l-1} - \alpha_{i-1} x_l = 0
\]
for $1 \leq i \leq \nu$ and $1 \leq l \leq k$. 
\end{lemma}

\begin{proof}
As in the proof of Lem.~\ref{j>0} we start with a $0$-cochain $(\sigma_U, \sigma_V)$ having $\sigma_U = \sum_{i=0}^\infty\sum_{l=0}^\infty \sigma_{il} z^l u^i$, and we look for the corresponding $\sigma_V$ making this a $0$-cocycle. Changing coordinates, we have
\[
T \sigma = z^j \sum_{i=0}^\infty \sum_{l=0}^\infty \sigma_{il} z^l u^i.
\]
Here, since $j > 0$, 
terms on $u^0$ are not holomorphic on $V$. To have holomorphicity, we need
only terms $z^ru^s$ with $r\leq ks$, we are thus looking to satisfy the 
condition $j+l\leq ki$. Thus, we arrive at the expression 
 $T\sigma= \sum_{i=0}^\infty\sum_{l=0}^{ki-j} \sigma_{il}z^{j+l}u^i$
which terms are nonzero only when $ki-j\geq 0$, that is, 
$i\geq \lfloor j/k \rfloor $. Using the notation set up just 
above for the Euclidean algorithm, we have $q=\lfloor j/k\rfloor$, 
and we are searching for holomorphic 
terms in the expression 
 $T\sigma= \sum_{i= q}^\infty\sum_{l=0}^{ki-j} \sigma_{il}z^{j+l}u^i$.
Note that for $i=q$ we have $0\leq l \leq kq-j = \nu$, thus, we may 
rewrite
\begin{align*}
T \sigma &= \sum_{l=0}^{\nu} \sigma_{ql}z^{j+l}u^q
 + \sum_{i= q+1}^\infty\sum_{l=0}^{ki-j} \sigma_{il}z^{j+l}u^i\text{,}
\intertext{or equivalently}
T \sigma &= \sum_{l=q}^{\nu} \sigma_{ql}z^{j+l}u^q
 + \sum_{i= 1}^\infty\sum_{l=0}^{\nu+ki} \sigma_{il}z^{j+l}u^{q+i}.
\intertext{Thus}
\sigma &= \sum_{l=0}^{\nu} \sigma_{ql}z^{l}u^q
 + \sum_{i= 1}^\infty\sum_{l=0}^{\nu+ki} \sigma_{il}z^{l}u^{q+i},
\end{align*}
and the second sum on the right-hand side has only terms that can be obtained 
from the first sum via multiplying by $(z^{\leq k}u)^i$ which 
are global holomorphic functions on $Z_k$. Hence 
$\H^0(Z_k,{\mathcal O}(-j))$ is generated by 
the monomials $\alpha_i = z^i u^q$, for $0 \leq i \leq \nu$.
These satisfy the relations
\begin{flalign*}
&& \alpha_1 x_0 = \eqmakebox[alpha][r]{$\alpha_0 x_1$} &= \eqmakebox[q][r]{$z u^q$} && \\
&& \alpha_2 x_0 = \alpha_1 x_1 = \eqmakebox[alpha][r]{$\alpha_0 x_2$} &= \eqmakebox[q][r]{$z^2 u^q$} && \\[-.6em]
&& \vdots\hspace{3.8em}\vdots\hspace{.9em} & \hspace{3.7em}\vdots && \\[-.3em]
&& \alpha_\nu x_{k-1} = \eqmakebox[alpha][r]{$\alpha_{\nu-1} x_k$} &= \eqmakebox[q][r]{$z^{k+\nu-1} u^q$}. && \qedhere
\end{flalign*}
\end{proof}

\section*{Acknowledgements}
We thank Oren Ben-Bassat and Pushan Majumdar for helpful discussions and the referee for carefully reading our paper and making several useful suggestions.

S.B.\ acknowledges the generous support of the Studienstiftung des deutschen Volkes and the Max Planck Institute for Mathematics in Bonn, Germany. E.G.\ was partially supported by a Simons Associateship grant of ICTP and Network grant NT8 of the Office of External Activities at ICTP, Italy.

\end{document}